\documentclass[12pt,reqno]{amsart}
 
\numberwithin{equation}{section}
\usepackage[a4paper,margin=1in,footskip=0.25in]{geometry}

\usepackage{bbm}
\usepackage{amssymb}
\usepackage{amsmath}
\usepackage[all]{xy}
\usepackage{tikz}
\usepackage{amsfonts}
\usepackage{mathrsfs}
\usepackage{latexsym}
\usepackage{graphicx}
\usepackage{enumerate}

\usepackage{amscd,amssymb,amsmath,amsbsy,amsthm}
    \definecolor{linkred}{rgb}{0.7,0.2,0.2}
    \definecolor{linkblue}{rgb}{0,0.2,0.6}
    \definecolor{linkgreen}{rgb}{0,0.6,0.2}
\usepackage[colorlinks,backref,
    linkcolor=linkblue,
    citecolor=linkgreen,
    urlcolor=linkred]{hyperref}

\newtheorem{Th}{Theorem}[section]
\newtheorem{Lemma}[Th]{Lemma}
\newtheorem{Coro}[Th]{Corollary}
\newtheorem{Prop}[Th]{Proposition}

\newtheorem{Thm}{Theorem}[section]

\theoremstyle{definition}

\newtheorem{Eg}[Th]{Example}
\newtheorem{Rmk}[Th]{Remark}

\allowdisplaybreaks

\newcommand{\CP}{\mathrm{BPD}}
\newcommand{\wt}{\mathrm{wt}}
\newcommand{\id}{\mathsf{id}}
\newcommand{\supp}{\mathrm{supp}}
\def\S{\mathfrak{S}}

\newcommand{\BPD}[2][1pc]{%
\setlength{\unitlength}{#1}
\definecolor{lightcyan}{rgb}{0.8,1,1}%
\def\FF{%
    \qbezier(0.5,0)(0.5,0.2)(0.5,0.2)
    \qbezier(1,0.5)(0.8,0.5)(0.8,0.5)
    \qbezier(0.8,0.5)(0.5,0.5)(0.5,0.2)
}
\def\JJ{%
    \qbezier(0.5,1)(0.5,0.8)(0.5,0.8)
    \qbezier(0,0.5)(0.2,0.5)(0.2,0.5)
    \qbezier(0.5,0.8)(0.5,0.5)(0.2,0.5)
    }
\def\II{%
    \qbezier(0.5,0)(0.5,0.5)(0.5,1)}%
\def\HH{%
    \qbezier(0,0.5)(0.5,0.5)(1,0.5)}%
\def\XX{\NN\HH}
\def\NN{%
    \qbezier(0.5,0)(0.5,0.3)(0.5,0.3)
    \qbezier(0.5,1)(0.5,0.7)(0.5,0.7)}%
\def\BPDfr##1{%
\begin{picture}(1,1)%
    \linethickness{0.08\unitlength}
    ##1
    \thinlines
    \color{lightgray}%
    \put(0,0){\line(0,1){1}}%
    \put(1,0){\line(0,1){1}}%
    \put(0,0){\line(1,0){1}}%
    \put(0,1){\line(1,0){1}}%
\end{picture}}
\def\BPDfrc##1{\BPDfr{\put(0,0){\color{lightcyan}\rule{\unitlength}{\unitlength}}##1}}
\def\x{\BPDfrc{\XX}}
\def\f{\BPDfrc{\FF}}
\def\o{\BPDfrc{}}
\def\j{\BPDfrc{\JJ}}
\def\h{\BPDfrc{\HH}}
\def\i{\BPDfrc{\II}}
\def\O{\BPDfr{}}
\def\X{\BPDfr{\XX}}
\def\BX{\BPDfr{\II\HH}}
\def\F{\BPDfr{\FF}}
\def\J{\BPDfr{\JJ}}
\def\H{\BPDfr{\HH}}
\def\I{\BPDfr{\II}}
\def\B{\BPDfr{\FF\JJ}}
\def\|{
    \begin{picture}(0,1)%
        \linethickness{0.08\unitlength}
        \thinlines
        \color{lightgray}%
        \put(0,0){\line(0,1){1}}%
    \end{picture}}
\def\_{
    \begin{picture}(0,1)%
        \linethickness{0.08\unitlength}
        \thinlines
        \color{lightgray}%
        \put(0,0){\line(1,0){1}}%
    \end{picture}}
\def\^{
    \begin{picture}(0,1)%
        \linethickness{0.08\unitlength}
        \thinlines
        \color{lightgray}%
        \put(0,1){\line(1,0){1}}%
    \end{picture}}
\def\M##1{\begin{picture}(1,1)%
    \put(0,0.2){\makebox[\unitlength]{$##1$}}
\end{picture}}
\begin{array}{@{\,}c@{\,}}
\def\dots{\scriptstyle\cdots}
{\def\arraystretch{0}
\setlength{\arraycolsep}{0pc}
\color{teal}
\begin{array}{@{}l@{}}%
#2\end{array}}
\end{array}}

\newcommand{\clan}[2][1.2pc]{%
\setlength{\unitlength}{\dimexpr#1/2}%
\def\drawclan##1{%
\expandafter%
    \drawclanbegin.##1%
    \drawclanend\drawclanendend}
\def\drawclanbegin.##1##2\drawclanendend{%
    \ifx\drawclanend##1%
        {\relax}%
    \else%
    \clandot{##1}%
    \expandafter%
        \drawclanbegin.##2\drawclanendend%
\fi}
\def\clandot##1{%
    \if+##1
    \begin{picture}(2,1)
        \linethickness{0.15\unitlength}
        \color{red}
        \qbezier(0.4,0.4)(1,0.4)(1.6,0.4)
        \qbezier(1,1.0)(1,0.4)(1,-.2)
    \end{picture}
    \else\if-##1
    \begin{picture}(2,1)
        \linethickness{0.15\unitlength}
        \color{blue}
        \qbezier(0.5,0.4)(1,0.4)(1.5,0.4)
    \end{picture}
    \else\if.##1
    \begin{picture}(2,1)
        \color{teal}
        \put(1,0.4){\circle*{0.7}}
    \end{picture}
    \else
        \def\inpA{##1}\def\inpB{\dots}%
    \ifx\inpA\inpB%
    \begin{picture}(2,1)
        \color{teal}
        \put(0.5,0.4){\circle*{0.3}}
        \put(1,0.4){\circle*{0.3}}
        \put(1.5,0.4){\circle*{0.3}}
    \end{picture}
    \else
    \def\inpA{##1}\def\inpB{}%
    \ifx\inpA\inpB%
        \relax%
    \else
        \def\inpA{##1}\def\inpB{\pm}%
    \ifx\inpA\inpB%
    \begin{picture}(2,1)
        \linethickness{0.15\unitlength}
          \color{red}
          \qbezier(0.4,0.4)(1,0.4)(1.6,0.4)
          \qbezier(1,1.0)(1,0.4)(1,-.2)
          \color{blue}
          \qbezier(0.4,-.2)(1,-.2)(1.6,-.2)
    \end{picture}
    \else
        \def\inpA{##1}\def\inpB{\mp}%
    \ifx\inpA\inpB%
    \begin{picture}(2,1)
        \linethickness{0.15\unitlength}
        \color{red}
        \qbezier(0.4,0.4)(1,0.4)(1.6,0.4)
        \qbezier(1,1.0)(1,0.4)(1,-.2)
        \color{blue}
        \qbezier(0.4,1.0)(1,1.0)(1.6,1.)
    \end{picture}
    \else
    \begin{picture}(2,1)
        \color{white}\linethickness{4pt}
        \qbezier(1,0.4)({\numexpr(##1)+1},{\numexpr(##1)+1})({\numexpr(##1)*2+1},0.4)%
        \color{black}\linethickness{0.8pt}
        \qbezier(1,0.4)({\numexpr(##1)+1},{\numexpr(##1)+1})({\numexpr(##1)*2+1},0.4)%
        \color{teal}
        \put(1,0.4){\circle*{0.7}}
    \end{picture}
    \rule{0pc}{\dimexpr\unitlength*(##1)/2+\unitlength}%
    \fi\fi\fi\fi\fi\fi\fi%
}%
{\drawclan{#2}}%
}%

\DeclareFontEncoding{LS1}{}{}
\DeclareFontSubstitution{LS1}{stix}{m}{n}
\DeclareSymbolFont{stixletters}{LS1}{stix}{m}{it}
\DeclareMathAccent{\cev}{\mathord}{stixletters}{"91}
\DeclareMathAccent{\vec}{\mathord}{stixletters}{"92}

\begin{document}

\author{Yiming Chen}
\address{Department of Mathematics, 
Sichuan University, Chengdu, Sichuan 610065, P.R. China}
\email{ym\_chen@stu.scu.edu.cn, fan@scu.edu.cn, yaom@stu.scu.edu.cn}

\author{Neil J.Y. Fan}

\author{Rui Xiong}
\address[Rui Xiong]{Department of Mathematics and Statistics, University of Ottawa, 150 Louis-Pasteur, Ottawa, ON, K1N 6N5, Canada}
\email{rxion043@uottawa.ca}

\author{Ming Yao}

\def\ww{{\mathbf{w}}}
\def\uu{{\mathbf{u}}}
\def\vv{{\mathbf{v}}}
\def\ow{{\overline{w}}}
\def\ou{{\overline{u}}}
\def\ov{{\overline{v}}}
\def\word{\mathfrak{w}}
\def\calO{\mathcal{O}}
\def\pt{\mathsf{pt}}
\def\s{\mathfrak{s}}

\title[BPDs \& Clans]
{Bumpless Pipe Dream Fragments\\
---Equivariant Geometry of Clans}
 
\begin{abstract}
In this paper, we establish a new geometric setting for bumpless pipe dreams and double Schubert polynomials.
Building on the notion of bumpless pipe dream fragments, we define clan polynomials as their weight generating functions.
It turns out that clan polynomials arise naturally in the  equivariant geometry  of ($GL_p\times GL_q$)-orbits over the flag variety $Fl_{p+q}$ parametrized by $(p,q)$-clans. 
Furthermore, we show that the coefficients in  the equivariant Schubert expansion of the fundamental classes of  ($GL_p\times GL_q$)-orbit closures are exactly clan polynomials, which  resolves an open problem posed by Wyser and Yong. 
 
\end{abstract}

\maketitle

\section{Introduction}

A bumpless pipe dream (BPD) associated with a permutation $w\in S_n$ is a tiling of the $n\times n$  grid with the following tiles
\begin{align}\label{tiles}
\BPD[1.6pc]{\O}\ \BPD[1.6pc]{\H}\ \BPD[1.6pc]{\I}\ \BPD[1.6pc]{\X}\ \BPD[1.6pc]{\J}\ \BPD[1.6pc]{\F}
\end{align}
such that there are exactly $n$  non-overlapping pipes, labeled $1,\ldots,n$, each beginning at the south boundary  and exiting at the east boundary. 
BPDs were introduced by Lam, Lee, and Shimozono \cite{LLS18} to provide a  combinatorial model for double Schubert polynomials, and Weigandt \cite{Weigandt2021}  further showed that they can also serve as a combinatorial model for double Grothendieck polynomials by introducing  bumping tiles, see also Lascoux \cite{Las}.

In this paper, we reveal a new geometric interpretation of bumpless pipe dreams and double Schubert polynomials through the geometry of ($GL_p \times GL_q$)-orbit closures on the flag variety $Fl_{p+q}$.  
Observe that $(GL_{p+q}, GL_p \times GL_q)$ forms a symmetric pair. 
As shown by Vogan \cite{Vogan}, the geometry of these orbit closures is deeply intertwined with the representation theory of real groups; see also Lusztig--Vogan \cite{LV} for further developments.

The main combinatorial structure of this paper, which we call BPD fragments, is adapted from  BPDs supported on partitions considered by Weigandt \cite[Section 4.4]{Weigandt2021}; see the left figure \eqref{eq:BPDclaneg} for an example. 
The novel observation of this paper is that BPD fragments are naturally indexed by \emph{clans}, rather than permutations. 
For two positive integers $p$ and $q$, a ($p,q$)-clan is a partial matching  on $p+q$ nodes 
with unmatched nodes colored by $\clan{+}$ or $\clan{-}$, such that the number of $\clan{+}$ minus the number of $\clan{-}$ equals $p-q$, see the right figure in \eqref{eq:BPDclaneg} for an example.
\begin{equation}\label{eq:BPDclaneg}
\def\m#1{\makebox[1.2pc]{$#1$}}
\BPD{
\O\O\O\F\H\H\\
\O\F\H\X\H\M{}\\
\O\I\O\I\F\M{}\\
\F\J\F\X\J\M{}\\
\I\M{}\M{}\M{}\M{}\M{}\\
}\qquad 
\begin{matrix}
&\clan{6-84+-..-+.}\\
&\m{1}\m{2}\m{3}\m{4}\m{5}\m{6}\m{7}\m{8}\m{9}\m{10}\m{11}
\end{matrix}
\end{equation}

For a BPD fragment
, we can  read off its indexing clan by the following procedure. Pull straight the boundary of the  Young diagram  starting from the northeast corner to the southwest corner, replacing each step by a node. If there is a pipe connecting two steps, then draw an arc connecting them. Replace the unmatched horizontal steps by $\clan{-}$ and vertical steps by $\clan{+}$. Conversely, given a clan $\gamma$, we can construct a Young diagram $\lambda(\gamma)$ by corresponding the left endpoints and $\clan{+}$ (resp., right endpoints and $\clan{-}$) of $\gamma$ to vertical (resp., horizontal) steps of the southeast boundary of $\lambda(\gamma)$.  For instance, the left BPD fragment  corresponds to the right clan in  \eqref{eq:BPDclaneg}.

Let $\CP(\gamma)$ denote the set of all BPD fragments of a clan $\gamma$, and denote $\lambda(\gamma)$ by the Young diagram of  $\gamma$.  Notice that different clans may have the same Young diagram.  For instance, there are eleven BPD fragments for the clan in \eqref{eq:BPDclaneg}, all can be obtained by applying droop operations on the Rothe BPD fragment (the leftmost one). 
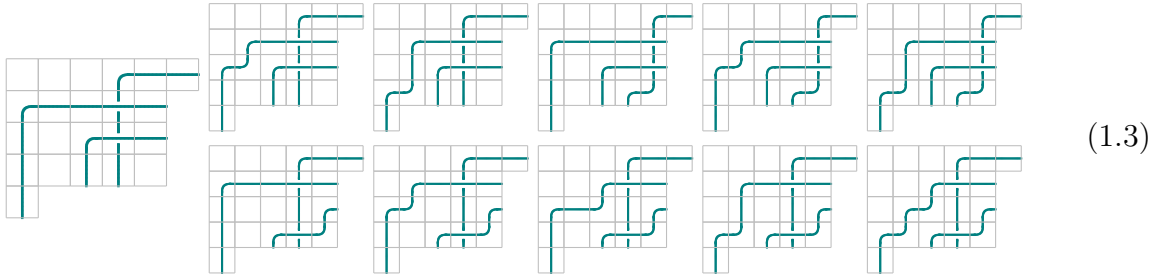
\begin{equation}\label{eq:example of droop}
\BPD{
\O\O\O\F\H\H\\
\F\H\H\X\H\M{}\\
\I\O\F\X\H\M{}\\
\I\O\I\I\O\M{}\\
\I\M{}\M{}\M{}\M{}\M{}}
\begin{matrix}
\BPD[0.8pc]{
\O\O\O\F\H\H\\
\O\F\H\X\H\M{}\\
\F\J\F\X\H\M{}\\
\I\O\I\I\O\M{}\\
\I\M{}\M{}\M{}\M{}\M{}}
\BPD[0.8pc]{
\O\O\O\F\H\H\\
\O\F\H\X\H\M{}\\
\O\I\F\X\H\M{}\\
\F\J\I\I\O\M{}\\
\I\M{}\M{}\M{}\M{}\M{}}
\BPD[0.8pc]{
\O\O\O\O\F\H\\
\F\H\H\H\X\M{}\\
\I\O\F\H\X\M{}\\
\I\O\I\F\J\M{}\\
\I\M{}\M{}\M{}\M{}\M{}}
\BPD[0.8pc]{
\O\O\O\O\F\H\\
\O\F\H\H\X\M{}\\
\F\J\F\H\X\M{}\\
\I\O\I\F\J\M{}\\
\I\M{}\M{}\M{}\M{}\M{}}
\BPD[0.8pc]{
\O\O\O\O\F\H\\
\O\F\H\H\X\M{}\\
\O\I\F\H\X\M{}\\
\F\J\I\F\J\M{}\\
\I\M{}\M{}\M{}\M{}\M{}}\\[-0.7pc]\\
\BPD[0.8pc]{
\O\O\O\F\H\H\\
\F\H\H\X\H\M{}\\
\I\O\O\I\F\M{}\\
\I\O\F\X\J\M{}\\
\I\M{}\M{}\M{}\M{}\M{}}
\BPD[0.8pc]{
\O\O\O\F\H\H\\
\O\F\H\X\H\M{}\\
\F\J\O\I\F\M{}\\
\I\O\F\X\J\M{}\\
\I\M{}\M{}\M{}\M{}\M{}}
\BPD[0.8pc]{
\O\O\O\F\H\H\\
\O\O\F\X\H\M{}\\
\F\H\J\I\F\M{}\\
\I\O\F\X\J\M{}\\
\I\M{}\M{}\M{}\M{}\M{}}
\BPD[0.8pc]{
\O\O\O\F\H\H\\
\O\F\H\X\H\M{}\\
\O\I\O\I\F\M{}\\
\F\J\F\X\J\M{}\\
\I\M{}\M{}\M{}\M{}\M{}}
\BPD[0.8pc]{
\O\O\O\F\H\H\\
\O\O\F\X\H\M{}\\
\O\F\J\I\F\M{}\\
\F\J\F\X\J\M{}\\
\I\M{}\M{}\M{}\M{}\M{}}
\end{matrix}
\end{equation}

For a $(p,q)$-clan $\gamma$, define the \emph{clan polynomial}
\begin{align}
\s_{\gamma}(x)=\s_{\gamma}(x_1,\ldots,x_{n}):=\sum_{\pi\in \CP(\gamma)}\wt(\pi),
\end{align}
where $n=p+q$ and 
\[
\wt(\pi)=\prod_{\text{$(i,j)$ is an empty tile $\BPD[0.6pc]{\O}$}}
\bigl(x_i-x_{n-j+1}\bigr).
\]
For example, the left BPD fragment in \eqref{eq:BPDclaneg} has weight 
$$(x_{1}-x_{11})(x_{1}-x_{10})(x_{1}-x_{9})
(x_{2}-x_{11})
(x_{3}-x_{11})
(x_{3}-x_{9}).$$

When the first $p$ (resp., last $q$) nodes are either left (resp., right) endpoints or $\clan{+}$ (resp., $\clan{-}$),  define a partial permutation $w_{\gamma}$ as follows: for an $(i,j)$-matching in $\gamma$, let $w_{\gamma}(i)=p+q+1-j$. Then the clan  polynomial 
$\mathfrak{s}_\gamma(x_1,\ldots,x_{p+q})$ reduces to the double Schubert polynomial $\S_{w_{\gamma}}(x_1,\ldots,x_p;y_1,\ldots,y_q)$ after changing $x_j$ to $y_{p+q+1-j}$ for $p+1\le j\le p+q$, see \cite{Kuntson,cocoa}.

The concept of clans arises from the study of orbits of the subgroup $GL_p\times GL_q$ of $GL_{p+q}$ on the flag variety $Fl_{p+q}$. 
By \cite{MO90}, each ($p,q$)-clan $\gamma$ indexes a 
($GL_p\times GL_q$)-orbit of the flag variety $Fl_{p+q}$, whose closure is denoted by $Y_\gamma$. 
Since the standard maximal torus of $GL_n$ is contained in $GL_p\times GL_q$, we are safe to talk about 
the equivariant fundamental class 
$[Y_\gamma]_T\in H_T^*(Fl_{p+q})$
in the torus equivariant cohomology of flag varieties. 
A nicely behaved polynomial representative of $[Y_\gamma]_T$, denoted as $\Upsilon_\gamma(x;y)$, was constructed by Wyser and Yong \cite{WY} via divided difference operators.

Our first main result is to establish a relation between clan polynomials and the localization of   $[Y_\gamma]_T$ at certain torus fixed point $v_\gamma$ defined below. 
More precisely, we give an explicit expression of 
$$
[Y_\gamma]_T|_{v_\gamma}=
\Upsilon_\gamma(x;y)|_{v_\gamma}
:= \Upsilon_{\gamma}(y_{v_\gamma(1)},\ldots,y_{v_{\gamma}(n)}; y_1,\ldots,y_n)$$
in terms of clan polynomials.

Each $(p,q)$-clan $\gamma$ can be associated with a $p$-inverse Grassmannian $v_{\gamma}$ and  a $q$-inverse Grassmannian $u_{\gamma}$  as follows. Firstly, assign each $\clan{-}$ and left endpoint  with $1,2,\ldots, q$ from  left to right, and then assign each $\clan{+}$ and right endpoint with $q+1,q+2,\ldots,n$ from left  to right. Then $u_\gamma$ is obtained by reading the assigned labels of $\gamma$ from left to right. Similarly,  
assign each $\clan{+}$ and left endpoint  with $1,2,\ldots,p$ from left to right, and assign $\clan{-}$ and right endpoint  with  $p+1,p+2,\ldots,n$  from
left to right, then $v_{\gamma}$ is obtained by reading the assigned labels of $\gamma$ from  left to right. 
It is easy to see that, $v_\gamma=\id$ if and only if all the $\clan{+}$ and left endpoints are to the left of all the $\clan{-}$ and right endpoints, equivalently, $\lambda(\gamma)$ is the whole $p \times q$ rectangle. 
For example, let $\gamma$ be the right clan of \eqref{eq:BPDclaneg}, then 
\begin{equation}\label{eq:egofugammavgamma}
\def\m#1{\makebox[1.2pc]{$#1$}}
\begin{aligned}
\gamma&=\,\clan{6-84+-..-+.}\\
u_\gamma &= \,\m{1}\m{2}\m{3}\m{4}\m{7}\m{5}\m{8}\m{9}\m{6}\m{10}\m{11} 
\\
v_\gamma &= \,\m{1}\m{6}\m{2}\m{3}\m{4}\m{7}\m{8}\m{9}\m{10}\m{5}\m{11} 
\end{aligned}
\end{equation}

Now we can state our first main result.
 
\begin{Thm}\label{th:localization}
For any $(p,q)$-clan $\gamma$, the localization
\begin{equation}\label{eq:clanlocThmA}
[Y_\gamma]_T|_{v_\gamma}=
\Upsilon_{\gamma}(x;y)|_{v_\gamma}= 
\s_{\gamma}(y)\cdot 
\prod_{(i,j)\in \overline{\lambda(\gamma)}}(y_{n-j+1}-y_i).
\end{equation}
where $n=p+q$ and $\overline{\lambda(\gamma)}$ is the complement of $\lambda(\gamma)$ in the $p\times q$ rectangle.
\end{Thm}

For example, for the clan $\gamma$ in  \eqref{eq:egofugammavgamma}, we have
$$[Y_{\gamma}]_T|_{v_\gamma}
=\s_{\gamma}(y)\cdot (y_6-y_2)(y_6-y_3)(y_6-y_4)(y_6-y_5)
(y_7-y_5)(y_8-y_5)(y_9-y_5)(y_{10}-y_5).$$  In Appendix \ref{appendix}, we  display the partial order on all the 24 $(2,2)$-clans under the containment of their $K$-orbit closures, and compute all the localizations of $\Upsilon_{\gamma}(x;y)$ at $v_\gamma$.
The essential step of the proof of Theorem \ref{th:localization} is Proposition \ref{prop:orbit_to_matrix_schubert}, where we relate $Y_\gamma$ with matrix Schubert varieties \cite{Fulton92,MK05}. 

Note that the localization formula \eqref{eq:clanlocThmA} is Graham positive \cite{Graham01} in the following sense 
$$
\frac{\Upsilon_{\gamma}(x;y)|_{v_\gamma}}{\prod_{(i,j)\in \overline{\lambda(\gamma)}}(y_{n-j+1}-y_i)}
=\mathfrak{s}_{\gamma}(y)\in \mathbb{Z}_{\geq 0}[y_i-y_j]_{i<j}.$$
We derive an inductive formula for the localization of $[Y_\gamma]_T$ at arbitrary fixed points (Proposition \ref{lem:indforloc}), 
but  do not observe a positive pattern as above. 
Besides  the localization in Theorem \ref{th:localization}, further properties of the localization are proved in  Proposition \ref{prop:localization}. 
For example, $v=v_\gamma$ is the smallest element (in the Bruhat order) of the support of $[Y_\gamma]_T|_v$, i.e., the permutation $v\in S_n$ with $[Y_\gamma]_T|_v=\Upsilon_\gamma(x;y)|_{v}\neq 0$. 

As an application of Theorem \ref{th:localization}, we provide an answer to the question posed by
Wyser and Yong \cite[Question 2]{WY}, which asks  a manifestly nonnegative combinatorial rule for 
the equivariant Schubert expansion of the fundamental class $[Y_{\gamma}]_T$, namely, the expansion of $\Upsilon_\gamma(x;y)$ into  $\S_w(x,y)$.
We show that the coefficients in the equivariant Schubert expansion of $\Upsilon_\gamma(x;y)$ are exactly clan polynomials.

\begin{Thm}\label{th:main}
    For any $(p,q)$-clan $\gamma$ with $p+q=n$, the coefficients $c_{\gamma,w}(y)$ in the expansion
\[\Upsilon_\gamma(x;y) = \sum_{w \in S_n} c_{\gamma,w}(y) \S_w(x;y)\]
are given by 
\begin{align}\label{eq:clanpoly}
    c_{\gamma,w}(y)=
\begin{cases}
    \s_{\gamma'}(y),\  \text{if}\  \gamma'=w*\gamma,\ v_{\gamma'}=\id\  \text{and}\ \ell(\gamma')-\ell(\gamma)=\ell(w),\\
        0,\ \qquad \text{otherwise}.\\ 
\end{cases}
\end{align}
where $w*\gamma=s_{i_1}*\cdots *s_{i_l}*\gamma$ for some reduced word $(i_1,\ldots,i_l)$ of $w$.
\end{Thm}

See Section \ref{subsect} for the definitions of length function on clans and the 0-Hecke  action of permutations on clans.
For example,  consider the ($3,2$)-clan $\gamma = \clan{1.+1.}$. 
The following diagram illustrates all clans of the form $\gamma'=w*\gamma$ for $w\in S_n$ and $\ell(\gamma')-\ell(\gamma)=\ell(w)$, where 
the framed are those with $v_{\gamma'}=\mathsf{id}$. 
$$
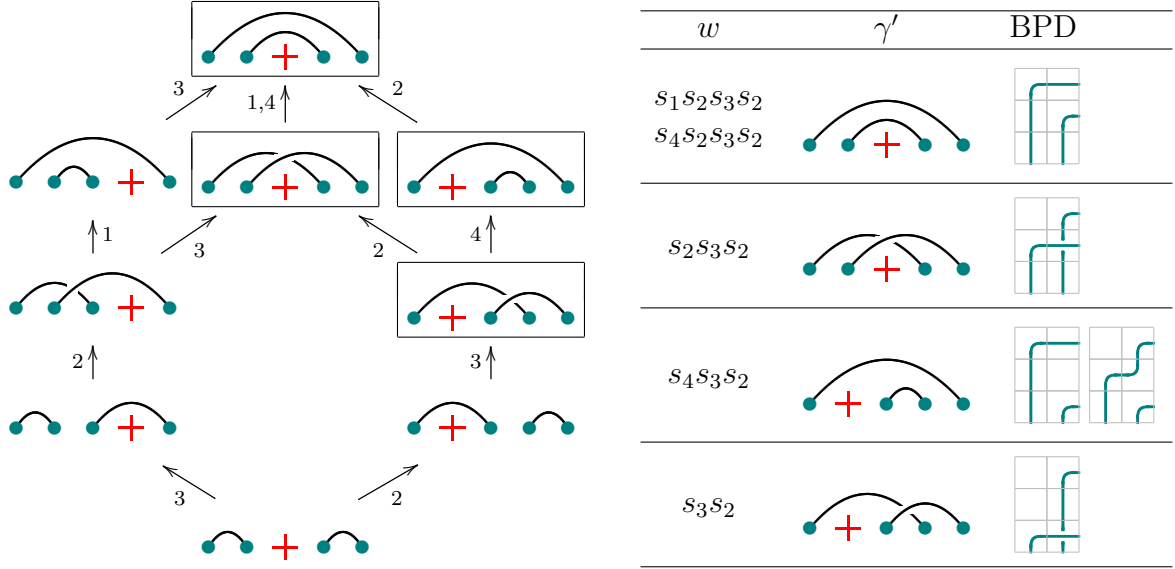
\def\myclan#1{{\hspace{-1pc}
    \begin{matrix}\\[-1pc]
    \clan{#1}\vphantom{\dfrac12}\\\end{matrix}
    \hspace{-1pc}}}
    \def\fmyclan#1{{\hspace{-.5pc}\begin{array}{|c|}\hline\\[-1pc]
    \hspace{-.5pc}\rule{0pc}{1.8pc}
    \clan{#1}\hspace{-.5pc}
    \\\hline\end{array}\hspace{-.5pc}}}
    \begin{matrix}
\xymatrix@=1pc{&
\fmyclan{42+..}&\\
\myclan{41.+.}\ar[ur]^3 &
\fmyclan{33+..}\ar[u]^{1,4} &
\fmyclan{4+1..}\ar[ul]_2 \\
\myclan{23.+.} \ar[u]_1\ar[ur]_3&&
\fmyclan{3+2..} \ar[ul]^2\ar[u]^4\\
\myclan{1.2+.}\ar[u]^2 &&
\myclan{2+.1.}\ar[u]^3 \\& 
\myclan{1.+1.}\ar[ul]^3\ar[ur]_2}
\end{matrix}\qquad 
    \def\myclan#1{{\hspace{-0pc}
    \begin{matrix}\\[-0.9pc]\\
    \clan{#1}\vphantom{\dfrac12}\\[-0.9pc]
    \\\end{matrix}
    \hspace{-0pc}}}
\begin{array}{ccl}\hline
w & \gamma' & \text{BPD}\\\hline
\begin{matrix}
s_1s_2s_3s_2\\
s_4s_2s_3s_2
\end{matrix}& \myclan{42+..}&
\BPD{\F\H\\\I\F\\\I\I}
\\\hline
s_2s_3s_2&\myclan{33+..}&
\BPD{\O\F\\\F\X\\\I\I}
\\\hline
s_4s_3s_2 & \myclan{4+1..}&
\BPD{\F\H\\\I\O\\\I\F}
\BPD{\O\F\\\F\J\\\I\F}
\\\hline
s_3s_2&\myclan{3+2..}&
\BPD{\O\F\\\O\I\\\F\X}\\\hline
\end{array}
$$
Then we have 
\begin{align*}
\Upsilon_{\gamma}(x;y)& = 
\mathfrak{S}_{s_1s_2s_3s_2}(x;y)+
\mathfrak{S}_{s_4s_2s_3s_2}(x;y)
+(y_1-y_5)\mathfrak{S}_{s_2s_3s_2}(x;y)\\
&\quad + (y_1+y_2-y_4-y_5)\mathfrak{S}_{s_4s_3s_2}(x;y)
 + (y_1-y_5)(y_2-y_5)\mathfrak{S}_{s_3s_2}(x;y).
\end{align*}

Note that the coefficients $c_{\gamma,w}(y)$ are manifestly positive in the sense of Graham \cite{Graham01}, i.e.,
$$c_{\gamma,w}(y)\in \mathbb{Z}_{\geq 0}[y_i-y_j]_{i<j}.$$
Actually, by the definition of clan  polynomials, each nonzero coefficient $c_{\gamma,w}(y)$ is a positive combination of squarefree monomials in $y_i-y_j$ for $0\leq i\leq p<j\leq n$. 
Recently, Graham, Jeon and Larson \cite{GJL} studied the Chern--Mather class of $Y_\gamma$, and conjectured them to be equivariantly positive. 

The non-equivariant Schubert expansion of $\Upsilon_\gamma(x;y)$ can be obtained by setting $y=0$. 
A general theorem of Brion \cite{Brion} shows that $c_{\gamma,w}(0)\in \{0,1\}$. 
Note that when $v_{\gamma}=\mathsf{id}$, 
the specialization 
$\mathfrak{s}_{\gamma}(0)$ is  equal to $1$ if and only if $\gamma$ is a rainbow. Recall that  a $(p,q)$-clan  is   called a {\it rainbow}, denoted as $\gamma_0$, if $\gamma_0$ consists of $\min\{p,q\}$  arcs matching $i$ with $p+q-i+1$ for $i\le\min\{p,q\}$  and $|p-q|$ signs in the middle. 
For example, the following diagram shows the only element in $\CP(\gamma_0)$, where $\gamma_0$ is the $(5,7)$-rainbow.
$$\BPD{
\F\H\H\H\H\H\H\\
\I\F\H\H\H\H\H\\
\I\I\F\H\H\H\H\\
\I\I\I\F\H\H\H\\
\I\I\I\I\F\H\H}\qquad 
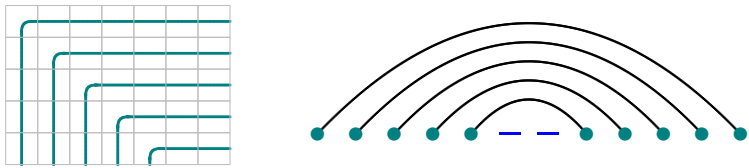
\begin{matrix}
\clan{{11}9753--.....}
\end{matrix}$$
By Theorem~\ref{th:main}, we have
\begin{align*}
    c_{\gamma,w}(0)=
    \begin{cases}
        1, \text{ if }w*\gamma=\gamma_0\text{ and } \ell(\gamma_0)-\ell(\gamma)=\ell(w),\\
        0,\text{ otherwise.}
    \end{cases}
\end{align*}

As an application of Theorem \ref{th:main}, we realize a family of equivariant Schubert structure constants in terms of clan polynomials. 

\begin{Thm}\label{thm:mainLRcoeff}
For any non-crossing clan $\gamma$ and any $v\in S_n$, assume that
\[
\S_{u_\gamma}(x;y)\S_{v}(x;y)=\sum d_{u_\gamma,
v}^{w}(y)\S_{w}(x;y). 
\]
Then the coefficient of $\S_{w_0v_\gamma}(x;y)$  is given by 
\[
d_{u_\gamma,v}^{w_0v_\gamma}(y)=c_{\gamma,w_0v}(\cev{y}),
\]  
where $c_{\gamma,w}(y)$ is  defined in \eqref{eq:clanpoly}. 
\end{Thm}

Pechenik and Weigandt \cite{Pandt} introduced back stable clans to  give a combinatorial rule for the structure constants in  the expansion of $\S_{u}(x)\S_{v}(x)$, where $u,v$ are  any general inverse Grassmannians. It  seems plausible to extend the results in \cite{Pandt} to the equivariant version  in terms of clan polynomials.

By \cite{Brion2}, localization could be used to test smoothness. 
Since the opposite Schubert variety indexed by $w \in S_n$  is known to be smooth at the minimal element in its support, one might expect that  $Y_\gamma$  is always smooth at  $v_\gamma$. However, this is not true in general. We provide a characterization of the smoothness of  $Y_\gamma$  at the point  $v_\gamma$  in terms of pattern avoidance of clans.
We say that a clan $\gamma$ contains the pattern $\sigma$ if there are nodes in $\gamma$ such that extracting these nodes from $\gamma$ in order gives the clan $\sigma$. We say $\gamma$ avoids $\sigma$ if $\gamma$ does not contain $\sigma$. For example, the right of \eqref{eq:BPDclaneg}
contains 
$\clan{41.-.}$ and 
avoids 
$\clan{4+1..}$.

\begin{Thm}\label{th:smoothness}
For a clan $\gamma$, the following statements are equivalent:
\begin{enumerate}[\quad \rm (1)]
\item the orbit closure $Y_\gamma$ is smooth at $v_\gamma$; 
\item  $\gamma$ has exactly one BPD fragment; 
\item  $\gamma$ avoids the following 5 patterns:
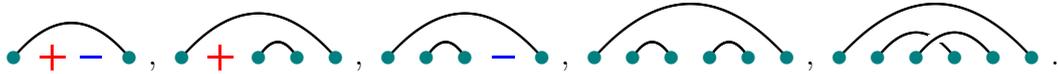
\[
\clan{3+-.}, \clan{4+1..}, \clan{41.-.}, \clan{51.1..},\clan{522...}.
\]
\end{enumerate}
\end{Thm}

When $w_\gamma$ is a permutation, i.e., $p=q$ and the first $p$ (resp., last $q$) nodes are all left (resp., right) endpoints, only the last pattern $\clan{522...}$ could occur. In this case, (3) is equivalent to that $w_\gamma$ is 
132-avoiding, and  the equivalence of (2) and (3) is well known, see for example, MacDonald \cite{MacDonald}. 

Let $\cev{\gamma}$ denote the clan obtained by interchanging unmatched $\clan{+}$ and $\clan{-}$ in $\gamma$. 
As will be shown in Proposition \ref{prop:localization}, we have $\Upsilon_{\gamma}(x;y)|_{w_0u_\gamma}=w_0(\Upsilon_{\cev{\gamma}}(x;y)|_{v_{\cev{\gamma}}})$.   
Similarly, the orbit closure $Y_\gamma$ is smooth at $w_0u_\gamma$ if and only if $Y_{\cev{\gamma}}$ is smooth at $v_{\cev{\gamma}}$, i.e., $\gamma$ avoids the patterns 
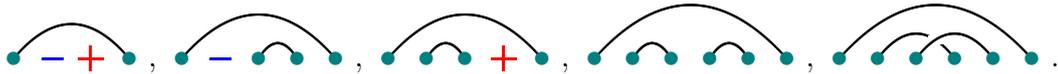
\[
\clan{3-+.}, \clan{4-1..}, \clan{41.+.}, \clan{51.1..},\clan{522...}.
\] 
By \cite{Wyser}, if $\gamma$ is non-crossing (or equivalently, avoids the pattern $\clan{22..}$), then $Y_\gamma$ is the Richardson variety $R_{v_\gamma}^{w_0u_\gamma}$. 
By \cite[Lemma 3.1]{WooWyser}, a Richardson variety $R_{u}^{v}$ is smooth if and only if it is smooth at $v$ and $u$. 
Theorem \ref{th:smoothness}  recovers the smoothness characterization of  $Y_\gamma$ when $\gamma$ is non-crossing, see McGovern \cite{McGovern} and  Woo and Wyser \cite{WooWyser}. 
That is, for  a non-crossing $\gamma$,  $Y_\gamma$ is smooth if and only if $\gamma$ avoids the following patterns: 
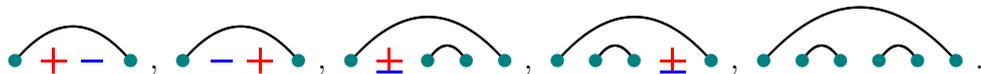
\[
    \clan{3+-.}, \clan{3-+.}, 
    \clan{4\pm 1..}, 
    \clan{41.\pm.}, 
    \clan{51.1..}.
\]

This paper is organized as follows. In Section \ref{sec:preliminaries}, we review the necessary background on $K$-orbit closures, $(p,q)$-clans and the combinatorics of Schubert polynomials.   In Section \ref{sec:localization}, we study the localization of $\Upsilon_\gamma(x;y)$. Section \ref{sec:ProofsThmABC} is devoted to proving Theorems \ref{th:localization}, \ref{th:main}, \ref{thm:mainLRcoeff}.  Finally, in Section \ref{sec:smoothness}, we study the smoothness of $Y_\gamma$ at $v_\gamma$ and prove Theorem \ref{th:smoothness}.
 
\subsection*{Acknowledgments}
The authors would like to thank 
William Graham, 
Zackary Hamaker,
Oliver Pechenik,
J\"org Sch\"urmann, 
Frank Sottile,
Changjian Su, 
Anna Weigandt, and 
Sylvester W. Zhang 
for helpful discussions and valuable comments, and thank 
Tsao-Hsien Chen, Xiaolong Pan, and 
Yaolong Shen for explaining the theory of symmetric pairs.

\section{Preliminaries}\label{sec:preliminaries}

\subsection{$K$-orbit closures}
\label{subsec:polyrepn}
Let $G=GL_n$. 
We denote 
$T$ the subgroup of diagonal matrices and 
$B$ the subgroup of upper triangular matrices. 
The classical flag variety 
$$Fl_n= \{0=V_0\subset V_1\subset\ldots\subset V_{n-1}\subset V_n=\mathbb{C}^n:\dim V_i=i\}$$
can be identified with $G/B$.
The torus fixed points of $Fl_n$ are in one-to-one correspondence to permutations in the symmetric group $S_n$. 
Precisely, for each permutation $w\in S_n$, we can construct a full flag $V^w_\bullet\in Fl_n$ given by 
$$V_i^w = V_{i-1}^w\oplus \mathbb{C}e_{w(i)},$$
where $e_1,\ldots,e_n$ are the standard basis of $\mathbb{C}^n$. 
Denote $\dot{w}\in GL_n$ by the permutation matrix of $w\in S_n$, that is, the $(w(i),i)$ entry of $\dot{w}$ equals $1$ and the other entries are $0$. 
Then we can identify $\dot{w}B/B=V_\bullet^w$ for $w\in S_n$. 
Let $B^-$ be the subgroup of lower triangular matrices. 
Denote the opposite Schubert cell and opposite Schubert variety of $w$ as
$$X_w^\circ = B^-\dot{w}B/B = \text{the $B^-$-orbit of $V^w_\bullet$},\qquad 
X_w = \overline{X_w^\circ}.$$
It is well known that $Fl_n$ can be decomposed into finite many $B^-$orbits, indexed by permutations $w\in S_n$, i.e., 
$$Fl_n = \bigsqcup_{w\in S_n} X_w^\circ.$$

Now fix $p+q=n$. 
We are going to recall an analogue of  Bruhat decomposition of $Fl_n$ into $GL_p\times GL_q$-orbits. 
Let $\theta:G\to G$ be the involution of conjugation by $\operatorname{diag}(\mathbf{1}_p,-\mathbf{1}_q)\in G$. The invariant subgroup of $\theta$ is 
$$K=G^\theta = \left[\begin{matrix}
GL_p&\\&GL_q
\end{matrix}\right]\cong GL_p\times GL_q.$$
The symmetric pair $(G,K)$ is known as type AIII in the classification; see for example \cite{Araki62}. 

For each ($p,q$)-clan $\gamma$, we construct a full flag $V_\bullet^\gamma\in Fl_n$ by  letting
$V^\gamma_i=V^\gamma_{i-1}\oplus \mathbb{C}\mathbf{v}^\gamma_i$, where 
\begin{itemize}
    \item $\mathbf{v}^\gamma_i = e_{v_\gamma(i)}+e_{v_\gamma(j)}$ if 
    the $i$-th node of $\gamma$ is a left endpoint matched with the $j$-th node; 
    \item $\mathbf{v}^\gamma_i = e_{v_\gamma(i)}$ if the 
    $i$-th node of $\gamma$ is a right endpoint or $\clan{\pm}$. 
\end{itemize}
Let $\dot{\gamma}$ denote the matrix $\big[\begin{matrix}
\mathbf{v}_1^\gamma & \mathbf{v}_2^\gamma
 & \cdots &
 \mathbf{v}_{n}^\gamma
\end{matrix}\big]$. 
In the running example \eqref{eq:egofugammavgamma}, we have 
$$
\def\m#1{\makebox[1.2pc]{$\def\arraystretch{0}\begin{array}{c}#1\end{array}$}}
\begin{aligned}
\gamma =&\,\clan{6-84+-..-+.}\\
\mathbf{v}^\gamma_i =&\, 
\m{e_1\\+\\e_8}
\m{e_6}\m{e_2\\+\\e_{11}}
\m{e_3\\+\\e_9}\m{e_4}\m{e_7}\m{e_8}\m{e_9}\m{e_{10}}\m{e_{5}}\m{e_{11}},
\end{aligned}\qquad 
\dot{\gamma}=
 \mbox{\tiny$
\left[\begin{array}{*{11}{c}}
1&0&0&0&0&0&0&0&0&0&0\\ 
0&0&1&0&0&0&0&0&0&0&0\\ 
0&0&0&1&0&0&0&0&0&0&0\\ 
0&0&0&0&1&0&0&0&0&0&0\\ 
0&0&0&0&0&0&0&0&0&1&0\\ 
0&1&0&0&0&0&0&0&0&0&0\\ 
0&0&0&0&0&1&0&0&0&0&0\\ 
1&0&0&0&0&0&1&0&0&0&0\\ 
0&0&0&1&0&0&0&1&0&0&0\\ 
0&0&0&0&0&0&0&0&1&0&0\\ 
0&0&1&0&0&0&0&0&0&0&1\\ 
\end{array}\right]
 $}.
$$

Denote the $K$-orbit  of $\gamma$ as
$$Y_\gamma^\circ = K\dot{\gamma}B/B=
\text{the $K$-orbit of $V^\gamma_\bullet$}.$$
Let $Y_\gamma = \overline{Y_{\gamma}^\circ}$ denote the closure of $Y_{\gamma}^\circ$.
By \cite[Theorem 3]{Matsuki}, the flag variety can be decomposed into finitely many $K$-orbits indexed by ($p,q$)-clans
$$Fl_n=\bigsqcup_{\gamma} Y_\gamma^\circ.$$
  
\subsection{Polynomial Representatives}\label{subsect}
Now we recall the polynomial representatives (under Borel's isomorphism) of equivariant fundamental classes
$$[X_w]_T\text{ and }[Y_\gamma]_T\in H_T^*(Fl_n)$$
for permutations $w\in S_n$ and ($p,q$)-clans $\gamma$.
Note that both are well-defined since $T\subset B^-$ and $T\subset K$, respectively. 

The equivariant Schubert class $[X_w]_T\in H_T^*(Fl_n)$ can be represented by the double Schubert polynomial $\mathfrak{S}_w(x;y)$, see \cite{LS82,AF23}. 
For $w\in S_n$,  $\mathfrak{S}_w(x;y)$ is recursively defined by
\begin{align*}
\S_{w_0}(x;y)& =\prod_{i+j\le n}(x_{i}-y_{j}),
&\hspace{-2pc}& \text{if $w_0=n\cdots 21$}\\
\S_{ws_i}(x;y)& =\partial_{i}\S_{w}(x;y),
&\hspace{-2pc}&\text{ if $ws_i<w$,} 
\end{align*}
where 
$$\partial_{i}f=\frac{f-f|_{x_{i}\leftrightarrow x_{i+1}}}{x_i-x_{i+1}}$$
is the divided difference operator.
For any $w$, let $\partial_{w}=\partial_{i_1}\cdots \partial_{i_\ell}$ for any reduced decomposition $w=s_{i_1}\cdots s_{i_\ell}$. 
Then we have 
$$\partial_{u}\S_w(x;y) = \begin{cases}
\S_{wu^{-1}}(x;y), & \text{if} \ \ell(wu^{-1})=\ell(w)-\ell(u),\\
0, & \text{otherwise}.
\end{cases}$$
 
Next, let us 
recall a well-behaved polynomial representative of $[Y_\gamma]_T$ constructed by Wyser and Yong \cite{WY}. 
 The length of a clan is defined as
$\ell(\gamma)=\dim Y_\gamma-\frac{p(p-1)}{2}-\frac{q(q-1)}{2}$, which can also be described as 
$$
\ell(\gamma) = \sum_{(i,j)\text{-matching}} (j-i) 
- \texttt{\#}\{\text{crossings in $\gamma$} \}.$$

There is a $0$-Hecke algebra action on clans. 
Denote the partial flag variety
$$Fl_{n,i}
=\{0\subset\cdots\subset V_{i-1}\subset V_{i+1}\subset \cdots\subset \mathbb{C}^n:\dim V_i=i\},$$
and the natural projection $\pi_i:Fl_n\to Fl_{n,i}$ by forgetting the $i$-th subspace. 
For $0\leq i\leq n-1$,  
$s_i* \gamma$
is exactly the clan such that  $Y^\circ_{s_i* \gamma}$ is dense in $\pi_i^{-1}(\pi_i(Y^\circ_{\gamma}))$.
The clan $s_i*\gamma$ can be obtained as follows. 
If the $i$-th and $(i+1)$-th nodes of $\gamma$ are of opposite signs, then add an arc between them to get $s_i* \gamma$. 
Otherwise, let $\gamma'$ be the clan obtained by swapping the $i$-th and $(i+1)$-st nodes of $\gamma$. If $\ell(\gamma')>\ell(\gamma)$, then $s_i* \gamma=\gamma'$, and $s_i* \gamma=\gamma$ otherwise. 
See the following figure as an illustration. 
 
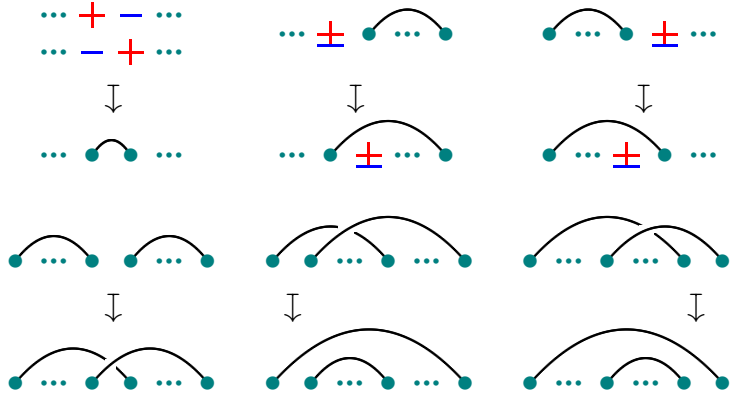
\def\mpto{\rotatebox{-90}{$\mapsto$}}
$$\begin{array}{ccc}
\begin{matrix}
\clan{\dots+-\dots}\\
\clan{\dots-+\dots}
\end{matrix}& 
\clan{\dots\pm2\dots.}&
\clan{2\dots.\pm\dots}\\[-1ex]
\mpto&\qquad\quad\mpto\hfill&\hfill\mpto\quad\qquad\\
\clan{\dots1.\dots}&
\clan{\dots3\pm\dots.}&
\clan{3\dots\pm.\dots}
\\\\
\clan{2\dots.2\dots.}&
\clan{34\dots.\dots.}&
\clan{4\dots3\dots..}\\[-1ex]
\mpto&\quad\mpto\hfill&\hfill\mpto\quad\,\,\\
\clan{3\dots3.\dots.}&
\clan{52\dots.\dots.}&
\clan{5\dots2\dots..}
\end{array}$$
We say that $s_i* \gamma $ covers $ \gamma$, denoted as  $s_i* \gamma > \gamma$, if  $\ell(s_i*\gamma)>\ell(\gamma)$. 
The transitive closure of the covering relations $s_i* \gamma > \gamma$ is called the weak Bruhat order on clans.

Given a $(p,q)$-clan $\gamma$ with $p+q=n$,  denote 
$\cev{y}$ the family of variables $(y_n,\ldots,y_2,y_1)$. 
The polynomial representative $\Upsilon_{\gamma}(x;y)$ of $[Y_\gamma]_T$ constructed by Wyser and Yong \cite{WY} is characterized by 
\begin{align*}
\Upsilon_{\gamma}(x;y)& =\S_{u_\gamma}(x;\cev{y})\cdot\S_{v_\gamma}(x;y),&\hspace{-2pc}&
\text{if $\gamma$ is matchless}\\
\Upsilon_{s_i* \gamma}(x;y) & = \partial_i \Upsilon_{\gamma}(x;y),&\hspace{-2pc}&
\text{if $s_i* \gamma> \gamma$}.
\end{align*}
Then we have 
\begin{equation}\label{eq:def of poly}
    \partial_u 
\Upsilon_\gamma(x;y)
=\begin{cases}
\Upsilon_{u*\gamma}(x;y), & \text{if} \
\ell(u*\gamma)-\ell(\gamma) =\ell(u),\\
0, & \text{otherwise}.
\end{cases}
\end{equation}

\subsection{Bumpless pipe dream and droops}
Now we generalize the droop operation in \cite[Section 5.2]{LLS18} and \cite{Weigandt2021} to BPD fragments. 
Let $\gamma$ be a $(p,q)$-clan with $p+q=n$, define the Rothe BPD fragment $D(\gamma)$ of $\gamma$ as follows. For an $(i,j)$-matching in $\gamma$, put a $\BPD{\F}$ at $(v_{\gamma}(i),n+1-v_{\gamma}(j))$, then draw a horizontal line to the right and a vertical line to the bottom of $\lambda(\gamma)$ from $(v_{\gamma}(i),n+1-v_{\gamma}(j))$. In particular, if $\lambda(\gamma)=p\times q$, i.e., $v_{\gamma}=\id$, $D(\gamma)$ is exactly the Rothe diagram of the partial permutation $w_{\gamma}$. It can be easily checked  that if $\gamma$ contains a pattern $\sigma$, then $D(\sigma)$ can be obtained from $D(\gamma)$ by deleting some rows and columns. 

Suppose that there is a rectangluar region with a $\BPD{\F}$ in the northwest corner, a $\BPD{\O}$ in the southeast corner, and no elbow tiles $\BPD{\F}$ or $\BPD{\J}$ inside. A \emph{droop} operation first changes the northwest $\BPD{\F}$ tile to a $\BPD{\O}$, and the southeast $\BPD{\O}$ tile to a $\BPD{\J}$, then draws a horizontal line to the left and a vertical line to the top at the new $\BPD{\J}$ tile until these lines reach the original pipes: 

$$\BPD{
\F\H\X\H\H\\
\I\O\I\O\O\\
\I\O\I\O\O
}\quad 
\rightarrow
\quad
\BPD{\O\O\I\O\F\\
\O\O\I\O\I\\
\F\H\X\H\J}$$

The following proposition is similar to \cite[Proposition 5.3]{LLS18}, except that  now we only allow droops inside the Young diagram $\lambda(\gamma)$. 
\begin{Prop} \label{prop:Anna BPD can be obtained by droops}
Any BPD fragment of a clan can be obtained from its Rothe BPD by applying a sequence of droops.  
\end{Prop}
See \eqref{eq:example of droop} for an example, all BPD fragments can be obtained by applying droop operations from the Rothe BPD (the leftmost one).

\section{Localization formulas}
\label{sec:localization}

\subsection{General localization formulas}
In this subsection, we will discuss the general localization formulas. 
For a permutation $u\in S_n$, the flag $V_\bullet^u$ is a $T$-fixed point. 
Thus we can consider the localization map 
$$-|_u : H_T^*(Fl_n)\longrightarrow 
H_T^*(\{V_\bullet^u\})\cong
H_T^*(\mathsf{pt})\cong 
\mathbb{Q}[y_1,\ldots,y_n].$$
 Combinatorially, this corresponds to specializing $x_i$ to $y_{u(i)}$ under the Borel isomorphism. 
For example, 
\begin{align*}
[X_w]_T|_u = \mathfrak{S}_w(x;y)|_{u} & := \mathfrak{S}_w(y_{u(1)},\ldots,y_{u(n)};
    y_1,\ldots,y_n)
\in \mathbb{Q}[y_1,\ldots,y_n],\\
[Y_\gamma]_T|_u=\Upsilon_\gamma(x;y)|_{u} & := \Upsilon_\gamma(y_{u(1)},\ldots,y_{u(n)};
    y_1,\ldots,y_n)
\in \mathbb{Q}[y_1,\ldots,y_n].
\end{align*}

Recall that $\cev{\gamma}$ is obtained from $\gamma$ by interchanging all unmatched $\clan{+}$ and $\clan{-}$ signs in $\gamma$.

\begin{Prop}\label{prop:localization}
For a $(p,q)$-clan $\gamma$ with $p+q=n$, the localization of $\Upsilon_{\gamma}(x;y)$ satisfies:
\begin{itemize}
\item[(1)]
For $v\in S_p\times S_q$, 
$
\Upsilon_{\gamma}(x;y)|_{vu}
=v(\Upsilon_{\gamma}(x;y)|_{u})$. 
\item[(2)] 
$
\Upsilon_{\cev{\gamma}}(x;y)|_{u}
= w_0(\Upsilon_{\gamma}(x;y)|_{w_0u})$. 
\item[(3)]
$\Upsilon_{\gamma}(x;y)|_{w_0u_\gamma}=w_0(\Upsilon_{\cev{\gamma}}(x;y)|_{v_{\cev{\gamma}}}). 
$
\item[(4)] 
$
\Upsilon_{\gamma}(x;y)|_{w} =0$ unless $v_\gamma\leq w\leq w_0u_\gamma$. 
\end{itemize}
\end{Prop}

\begin{proof}
Since $\S_{u_\gamma}(x;\cev{y})=\S_{u_\gamma}(x;y_n,\ldots,y_1)$ is symmetric in $y_{p+1},\ldots,y_n$  and $\S_{v_\gamma}(x;y)$ is symmetric in $y_1,\ldots,y_p$, the polynomial $\Upsilon_{\gamma}(x;y)$ is also symmetric in $y_{p+1},\ldots,y_n$ and $y_1,\ldots,y_p$, respectively. 
This proves (1).  

By \cite[Proposition 2.8]{WY}, we have $\Upsilon_{\cev{\gamma}}(x;y)=\Upsilon_{\gamma}(x,\cev{y})=\Upsilon_{\gamma}(x_1,\ldots,x_n;y_n,\ldots,y_1)$.
Thus
\[\Upsilon_{\cev{\gamma}}(x;y)|_{u}=\Upsilon_{\gamma}(y_{u(1)},\ldots,y_{u(n)};y_n\ldots,y_1)=w_0\cdot\Upsilon_{\gamma}(y_{w_0u(1)},\ldots,y_{w_0u(n)};y_1,\ldots,y_n),\]
which implies (2). 

By (2), we  have $\Upsilon_{\gamma}(x;y)|_{w_0u_\gamma}=w_0(\Upsilon_{\cev{\gamma}}(x;y)|_{u_\gamma}$), now (3) follows  from $u_\gamma=v_{\cev{\gamma}}$.  

To prove (4), let $\supp(\Upsilon_{\gamma}(x;y))
:=\{w\,|\,\Upsilon_{\gamma}(x;y)|_{w}\neq0\}$. We will show that $\supp(\Upsilon_{\gamma}(x;y))
\subseteq [v_\gamma,w_0u_\gamma]$ by induction. For a matchless $\tau$, by definition, we have
\[\Upsilon_{\tau}(x;y) = \S_{u_{\tau}}(x;\cev{y})\cdot 
    \S_{v_{\tau}}(x;y).\]
Since $\S_w(x;y)|_v=0$ unless $v\geq w$, we have $\supp(\Upsilon_{\tau}(x;y))\subseteq [v_\tau,w_0u_\tau]$. Now suppose $\gamma=s_i*\gamma'$ with $\supp(\Upsilon_{\gamma'}(x;y))\subseteq [v_{\gamma'},w_0u_{\gamma'}]$. By \cite[Proposition 2.2 (1)]{Kuntson}, for any $f\in \mathbb{Q}[x,y]$, we have 
\[\supp(\partial_if)\subseteq\supp(f)\cup\supp(f)\cdot s_i.\]
By definition, $\supp(\Upsilon_{s_{i}\gamma'})=\supp(\partial_{i}\Upsilon_{\gamma'})\subseteq[v_{\gamma'},w_0u_{\gamma'}]\cup[v_{\gamma'} s_i,w_0u_{\gamma'} s_i].$
It suffices to show that $v_{\gamma}=v_{s_i*\gamma'}=\min\{v_{\gamma'},v_{\gamma'}s_i\}$ and $u_{\gamma}=u_{s_i*\gamma'}=\min\{u_{\gamma'},u_{\gamma'}s_i\}.$ 

We only prove the case for $v_\gamma$, the analysis for $u_\gamma$ is similar. Consider the $i$-th and $(i+1)$-th position of $\gamma'$ and compare $v_{\gamma'}$ and $v_{\gamma'}s_i$ we obtain the following table. 
$$\begin{array}{c|c|c}\hline
\text{the $i$-th node of $\gamma'$} & 
\text{the ($i+1$)-th node of $\gamma'$} & 
\min(v_{\gamma'},v_{\gamma'}s_i)\\\hline
\clan{-}& \clan{+} & v_{\gamma'}s_i\\\hline
\clan{+} & \clan{-} & v_{\gamma'}\\\hline
\clan{-}& \text{left endpoint} & v_{\gamma'}s_i\\\hline
\clan{+}& \text{left endpoint} & v_{\gamma'}\\\hline
\text{right endpoint} & \clan{+} & v_{\gamma'}s_i\\\hline
\text{right endpoint} & \clan{-} & v_{\gamma'}\\\hline
\text{right endpoint} & \text{left endpoint} & v_{\gamma'}s_i\\\hline
\text{left endpoint} & \text{left endpoint} & v_{\gamma'}\\\hline
\text{right endpoint} & \text{right endpoint} & v_{\gamma'}\\\hline
\end{array}$$
\end{proof}
 
\begin{Prop}\label{lem:indforloc}
 Let $\gamma$ be a ($p,q$)-clan, and $w\in S_n$ be a permutation. We have 
$$
\Upsilon_{\gamma}(x;y)|_w = 
\Upsilon_{\gamma}(x;y)|_{ws_i}
+
\begin{cases}
(y_{w(i)}-y_{w(i+1)})\cdot\Upsilon_{s_i* \gamma}(x;y)|_{ws_i},&  \text{if} \
s_i* \gamma>\gamma,\\
0,& \text{otherwise}.
\end{cases}
$$
\end{Prop}

\begin{proof}
 If $s_i*\gamma>\gamma$, by definition,  we have
\begin{align*}
\left.\Upsilon_{s_{i}\ast\gamma}(x;y)\right|_{ws_{i}}
&= \left.\partial_{i}\Upsilon_{\gamma}(x;y)\right|_{ws_{i}} 
= \left.\frac{\Upsilon_{\gamma}(x;y)-\Upsilon_{\gamma}(x_{1},\ldots,x_{i+1},x_{i},\ldots,x_{n};y)}{x_{i}-x_{i+1}}\right|_{ws_{i}}\\
&= \frac{\left.\Upsilon_{\gamma}(x;y)\right|_{ws_{i}}-\left.\Upsilon_{\gamma}(x;y)\right|_{w}}{y_{w(i+1)}-y_{w(i)}}.
\end{align*}
If $s_i* \gamma =\gamma$, then  by \eqref{eq:def of poly}, $\partial_{i}\Upsilon_{\gamma}(x;y)=0$. Thus the assertion holds.
\end{proof}

By Proposition \ref{lem:indforloc}, to determine  $[Y_\gamma]_T|_w=\Upsilon_{\gamma}(x;y)|_w$ at any $w$, it suffices to determine them at a special $w$. 
We will deal with the case $w=\id$ in the next subsection. 
By Proposition \ref{prop:localization} (4), 
we have $[Y_\gamma]_T|_{\id}= 0$ unless $v_\gamma=\id$.

\subsection{Localization at $\id$}
In this subsection, we aim at finding a formula for $[Y_{\gamma}]_T|_\id=\Upsilon_{\gamma}(y;y)$ for clans $\gamma$ with $v_\gamma=\id$. 
Let 
 \[
K = \left[\begin{matrix}
     GL_{p} & 0 \\
     0 & GL_{q}
 \end{matrix}\right], \quad
 P^{-} = \left[\begin{matrix}
     GL_{p} & 0 \\
     \ast & GL_{q}
 \end{matrix}\right],\qquad 
N_P^- = \left[\begin{matrix}
     \mathbf{1}_{p} & 0 \\
     \ast & \mathbf{1}_{q}
 \end{matrix}\right],
\]
where $\ast$ denotes an arbitrary $q \times p$ matrix. 
We aim to decompose the localization at $\id$ into the following maps
\begin{equation}\label{eq:Tequivariant1}
\begin{matrix}
\xymatrix{&&&\\
H_T^*(G/B)\ar[r]^{f^*}
\ar`u[u]`[urrr]^{-|_{\id}}[rrr] & 
H_T^*(P^-B/B)\ar[r]^{g^*}
& 
H_T^*(N_P^-B/B)\ar[r]^{h^*}& H_T^*(1\cdot B/B)}
\end{matrix}
\end{equation}
where $f:P^-B/B\hookrightarrow G/B$, $g:N_P^-B/B\hookrightarrow P^-B/B$ and $h:1\cdot B/B\to N_P^-B/B$ are inclusions. 

Let us start by studying the property of $f$. 
Note that $P^- B/B\subset G/B$ is a $K$-equivariant open neighborhood of $1\cdot B/B$. 
Thus we have
\begin{align}\label{eq:STEP1_f}
    f^*([Y_\gamma]_T)= 
    \begin{cases}
        0, & \text{if} \ Y_\gamma\cap P^-B/B=\varnothing,\\
        [Y_\gamma\cap P^-B/B]_T, 
        & \text{if} \ Y_\gamma\cap P^-B/B\neq\varnothing. 
    \end{cases}
\end{align}
Therefore, $\Upsilon_{\gamma}(x;y)|_{\id}=0$ unless $Y_\gamma\cap P^-B/B\neq \varnothing$, or equivalently, $Y^\circ_\gamma\cap P^-B/B\neq \varnothing$. 
Under the assumption $v_\gamma=\id$, we have 
$\dot{\gamma}\in N_P^-\subset P^-$ by construction. 
As a result, $Y_\gamma\cap P^-B/B\neq \varnothing$. 

Next, let us study the behavior of $g$. 
We will show that $N_P^-B/B$ is a common transversal slice for all $Y^\circ_\gamma\subset P^-B/B$.

\begin{Prop}\label{prop:transslice}
The $N_P^-$-orbit $N_P^-B/B$ is closed in $P^-B/B$ and it intersects all $Y^\circ_\gamma\subset P^-B/B$ transversally. 
\end{Prop}

To this end, we need the following classical result of Richardson. 

\begin{Lemma}[Richardson \cite{Richardson92}]
\label{lem:Richardson}
Let $G'$ be a connected algebraic group, $H',K'$ and $L'$ be closed connected subgroups of $G'$, and let  $X = G'/L'$  satisfying the following conditions:
\begin{itemize}
    \item[(i)] $H'\cap K'$ is connected,
    \item[(ii)] The Lie algebras satisfy 
    $\operatorname{Lie}(H')+\operatorname{Lie}(K')=\operatorname{Lie}(G')$.
\end{itemize} 
Then for every $x\in X'$, the orbits $H'\cdot x$ intersects $K'\cdot x$ transversally, and the intersection $H'\cdot x\cap K'\cdot x$ is irreducible. 
\end{Lemma}

\begin{proof}[Proof of Proposition \ref{prop:transslice}]
In order to utilize the Lemma \ref{lem:Richardson}, set $G' = P^{-}$, $H' = N_{P}^{-}$, $K' = K$, and $L' = P^{-} \cap B$. Note that $L' = B_{p} \times B_{q}$ remains to be connected. The isomorphism
\[
   \phi: P^{-}B/B \to P^{-}/L'
 \]
restricts to isomorphisms $N^{-}_{P}B/B \cong N_{P}^{-}L'/L'$ and $KB/B \cong KL'/L'$.

Since the subgroup $N_P^-L'=L'\ltimes N_P^-$ is closed in $P^-=K'\ltimes N_P^-$, the $N_P^-$-orbit $N_P^-L'/L'$ is closed. 
The two conditions in Lemma \ref{lem:Richardson}  can be verified as follows:
\begin{itemize}
    \item[(i)] $N_{P}^{-}\cap K = \{\mathbf{1}_n\}$ is connected,
    \item[(ii)] the condition on Lie algebras holds because \begin{equation*}
        \operatorname{Lie}(N_{P}^{-})=\left[\begin{matrix}
        0 & 0\\
        \ast & 0
    \end{matrix}\right], \qquad \operatorname{Lie}(K)=\left[\begin{matrix}
        \mathfrak{gl}_{p} & \\
         & \mathfrak{gl}_{q}
    \end{matrix}\right], \qquad \operatorname{Lie}(P^-)=\left[\begin{matrix}
        \mathfrak{gl}_{p} & 0\\
        \ast & \mathfrak{gl}_{q}
    \end{matrix}\right].\qedhere\end{equation*}
\end{itemize}
\end{proof}
 
By Proposition \ref{prop:transslice}, we have 
\begin{equation}
\label{eq:STEP2_g}
g^*([Y_\gamma\cap P^-B/B]_T)
=[Y_\gamma\cap N_P^-B/B]_T
\end{equation}
by transversality.

In the last step, we need to study the variety $Y_\gamma\cap N_P^-B/B$ in $N_P^-B/B$. 
The key point is that, when $v_\gamma=\id$, the intersection $Y_\gamma\cap N_P^-B/B$ can be identified as the matrix Schubert variety for the partial permutation $w_\gamma$  determined by $w_{\gamma}(i)=n+1-j\in [q]:=\{1,2,\ldots,q\}$ for an $(i,j)$-matching in $\gamma$. 
Recall that $\dot{w}_{\gamma}$ is the $q\times p$ zero-one matrix with the $(n+1-j,i)$ entry equals $1$ and other entries equal $0$.

\begin{Prop}\label{prop:orbit_to_matrix_schubert}
For a ($p,q$)-clan $\gamma$ with $v_\gamma=\id$, the intersection 
$Y_\gamma\cap N_P^-B/B$  corresponds to the 
matrix Schubert variety $$\overline{\dot{w}_0B_q^- \dot{w}_{\gamma} B_p}$$ 
under the isomorphism 
$M_{q\times p}(\mathbb{C})\cong N_P^-\stackrel{\sim}\longrightarrow N_P^-B/B.$
\end{Prop}

The following lemma should be known to experts, but we do not find a proper reference. 

\begin{Lemma}\label{lem:Richardsonclosure}
Under the assumption of Lemma \ref{lem:Richardson}, assume further that $G'/L'$ has only finitely many $K'$-orbit. For a closed $H'$-orbit $\mathbb{O}$, we have
$$\mathbb{O}\cap \overline{K'\cdot y}=
\overline{\mathbb{O}\cap K'\cdot y}. $$
\end{Lemma}
\begin{proof}
The argument is adapted from the proof of \cite[Theorem 2.1]{Richardson92}. 
The inclusion ``$\supseteq$'' is clear. 
Since there are only finitely many $K'$-orbits, we find
$$\overline{K'\cdot y}= K'\cdot y_0\cup 
\cdots \cup K'\cdot y_m$$
where $y_0=y$, and $\dim K'\cdot y_i<\dim K'\cdot y$ for $1\le i\le m.$
So 
\begin{align*}
\mathbb{O}\cap \overline{K'\cdot y}& = (\mathbb{O}\cap K'\cdot y_0)\cup 
\cdots \cup (\mathbb{O}\cap K'\cdot y_m)\\
& = \overline{\mathbb{O}\cap K'\cdot y_0}\cup \cdots \cup \overline{\mathbb{O}\cap K'\cdot y_m}. 
\end{align*}
By Lemma \ref{lem:Richardson}, each intersection $\mathbb{O}\cap K'\cdot y_i$ is irreducible. 
So the irreducible components of $\mathbb{O}\cap \overline{K'\cdot y}$ are the maximal (with respect to inclusion) members among $\overline{\mathbb{O}\cap K'\cdot y_i}$ for $i=0,\ldots,m$. 
Note that any irreducible component $Z\subset \mathbb{O}\cap \overline{K'\cdot y}$ satisfies 
$$\dim Z\geq \dim \mathbb{O}+\dim K'\cdot y-\dim X.$$
But by transversality, $\overline{\mathbb{O}\cap K'\cdot y}$ is the only member with this property. So we have $\mathbb{O}\cap \overline{K'\cdot y}=
\overline{\mathbb{O}\cap K'\cdot y}$. 
\end{proof}

\begin{Rmk}
J\"org Sch\"urmann shared with us in private communication that Lemma \ref{lem:Richardsonclosure} can be generalized to Whitney b-regular stratifications 
\cite[Remark 4.3.6]{Sch}; see also \cite[Example 4.3.4]{Sch}. 
\end{Rmk}

\begin{proof}[Proof of Proposition \ref{prop:orbit_to_matrix_schubert}]
Thanks to Lemma \ref{lem:Richardsonclosure}, it suffices to prove that  
\[
Y^\circ_{\gamma} \cap N_{P}^{-}B/B
\]  
corresponds to the open matrix Schubert variety $B_{q} \dot{w}_0\dot{w}_{\gamma} B_{p}=\dot{w}_0B_q^-\dot{w}_{\gamma}B_p$. 
Recall that we can write $\dot{\gamma}$ in the following way 
\[
\dot{\gamma}=\left[\begin{matrix}
    \mathbf{1}_{p} & 0 \\
    \dot{w}_0\dot{w}_{\gamma} & \mathbf{1}_{q}
\end{matrix}\right]\in N_P^-.\]
Each element in the $K$-orbit  $Y_{\gamma}^\circ$   looks like
\[
\left[\begin{matrix}
    A & 0 \\
    0 & C
\end{matrix}\right]
\left[\begin{matrix}
    \mathbf{1}_{p} & 0 \\    \dot{w}_0\dot{w}_{\gamma} & \mathbf{1}_{q}
\end{matrix}\right] B / B =
\left[\begin{matrix}
    A & 0 \\
    C\dot{w}_0\dot{w}_{\gamma} & C
\end{matrix}\right] B / B,
\]  
where  $A \in GL_{p}$ and $C \in GL_{q}$.
Suppose that $Y^\circ_{\gamma} \cap N_{P}^{-}B/B\neq\emptyset$. Then there exist upper triangular matrices $E \in B_{p}$, $G \in B_{q}$, and $F \in M_{p \times q}$ such that  
\[
\left[\begin{matrix}
    A & 0 \\
    C\dot{w}_0\dot{w}_{\gamma} & C
\end{matrix}\right]
\left[\begin{matrix}
    E & F \\
    0 & G
\end{matrix}\right] =
\left[\begin{matrix}
    AE & AF \\
    C\dot{w}_0\dot{w}_{\gamma}E & C\dot{w}_0\dot{w}_{\gamma} F + CG
\end{matrix}\right]=
\left[\begin{matrix}
    \mathbf{1}_{p} & 0 \\
    \ast & \mathbf{1}_{q}
\end{matrix}\right].
\]  
Thus $A \in B_{p}$, $F = 0$, and $C \in B_{q}$. This implies  
\[Y^\circ_{\gamma} \cap N_{P}^{-}B/B = 
\left[\begin{matrix}
    \mathbf{1}_{p} & 0 \\
    B_{q} \dot{w}_0\dot{w}_{\gamma} B_{p} & \mathbf{1}_{q}
\end{matrix}\right] B / B.
\]  
Taking the lower left $q\times p$ submatrix,  we can identify $Y^\circ_\gamma\cap N_P^-B/B$ with the open matrix Schubert variety $B_{q} \dot{w}_0\dot{w}_{\gamma} B_{p}$. Since $\dot{w}_0B_q\dot{w}_0=B_q^-$, we have $B_{q} \dot{w}_0\dot{w}_{\gamma} B_{p}=\dot{w}_0B_q^-\dot{w}_{\gamma}B_p$.
\end{proof}

By \cite{Fulton92,MK05}, the polynomial representive of the matrix Schubert variety $\overline{\dot{w}_0B^-_{q} \dot{w}_{\gamma} B_{p}}$ is $\S_{w_{\gamma}}(x;\cev{y})$. 
By Proposition \ref{prop:orbit_to_matrix_schubert}, we have the following commutative diagram: 
\begin{equation}
\begin{matrix}
\xymatrix{
  H_{T_q\times T_p}^*(M_{q\times p}(\mathbb{C})) \ar[d]^{\wr} \ar[r]^-{\sim} & H_{T_q\times T_p}^*(\mathsf{pt})
  \ar[d]^{\wr}\\
   H_T^*(N_P^-B/B)\ar[r]^-{h^*} & H_T^*(\mathsf{pt})}
\end{matrix}
\end{equation}
As a result, we have 
\[h^*([Y_\gamma\cap N_P^-B/B]_T) =\S_{w_{\gamma}}(x;\cev{y})|_{\id}= \mathfrak{S}_{w_\gamma}(y,\cev{y}).\]
Since $\S_{w_{\gamma}}(x_1,\ldots,x_p;y_1,\ldots,y_q)=\s_{\gamma}(x_1,\ldots,x_p,y_q,\ldots,y_1),$ we have 
\[\S_{w_{\gamma}}(y;\cev{y})=\S_{w_{\gamma}}(y_1,\ldots,y_p;y_{p+q},\ldots,y_{p+1})=\s_{\gamma}(y_1,\ldots,y_p,y_{p+1},\ldots,y_{p+q})=\s_\gamma(y).\]
Consequently, \begin{equation}
\label{eq:STEP3_h}
h^*([Y_\gamma\cap N_P^-B/B]_T) = \mathfrak{S}_{w_\gamma}(y,\cev{y})=\s_\gamma(y).
\end{equation}

Combining 
\eqref{eq:STEP1_f},
\eqref{eq:STEP2_g},
\eqref{eq:STEP3_h}, we get the following. 

\begin{Prop}\label{prop:local at id}
For a ($p,q$)-clan $\gamma$ with $v_\gamma=\id$, we have 
$$[Y_\gamma]_T|_{\id} = \Upsilon_\gamma(x,y)|_{\id} = \s_{\gamma}(y).$$
\end{Prop}

\section{Proofs of Theorems  \ref{th:localization},
\ref{th:main} and 
\ref{thm:mainLRcoeff}}
\label{sec:ProofsThmABC}

\subsection{Proof of Theorem \ref{th:localization}}

Recall the process of reading off the underlying clan from a BPD fragment on the Young diagram $\lambda(\gamma)$. We pull straight the southeast boundary of $\lambda(\gamma)$  starting from the northeast corner to the southwest corner, and replace each step by a node. If there is a pipe connecting two steps, then draw an arc connecting the corresponding two nodes. And finally replace the unmatched horizontal steps by $\clan{-}$ and vertical steps by $\clan{+}$.

\begin{proof}[Proof of Theorem \ref{th:localization}]
We proceed by induction on the number of boxes of $\overline{\lambda(\gamma)}$.  
For the base case, when $\lambda(\gamma)$ is the full $p\times q$ rectangle, we have $v_{\gamma}=\id$. This case has already been established in Proposition \ref{prop:local at id}.  

Now assume the theorem holds for all clans $\theta$ such that $\texttt{\#}\overline{\lambda(\theta)}\le k$. Suppose that there is a clan $\gamma$ such that $\texttt{\#}\overline{\lambda(\gamma)}=k+1$. Then there exists a horizontal step, say the $t$-th step, of $\lambda(\gamma)$ such that the $(t+1)$-th step is a vertical step. Swapping the $t$-th step and  $(t+1)$-th step of $\lambda(\gamma)$ to obtain $\lambda(s_t*\gamma)$.
Then $s_t*\gamma>\gamma$ and   $\lambda(s_t*\gamma)$ is obtained from $\lambda(\gamma)$ by adding a new box $A$ at the position $(v_{\gamma}(t+1),\, n+1-v_{\gamma}(t))$. Clearly, $\texttt{\#}\overline{\lambda(s_t*\gamma)}=k$.
We are going to establish a bijection between $\CP(\gamma)$ and $\CP(s_t*\gamma)$. 
There are four cases. 
\begin{itemize}

\item If the $t$-th node of $\gamma$ is $\clan{-}$ and the ($t+1$)-th node is $\clan{+}$, then there are no pipes going through the upper edge and left edge of the added box $A$. We fill the added box $A$ with an SE elbow $\BPD{\F}$.
$$\clan{{}\dots\dots-+\dots\dots{}}
    \xrightarrow{\ s_t*-\ }
    \clan{{}\dots\dots1.\dots\dots{}}\qquad 
    \text{e.g. }
    \BPD{\M{\dots}\M{\dots}\M{\dots}\\
    \M{\dots}\I\O\\\M{\dots}\J \M{A}}\longrightarrow \BPD{\M{\dots}\M{\dots}\M{\dots}\\
    \M{\dots}\I\O\\\M{\dots}\J\F}.$$
    
\item If the $t$-th node of $\gamma$ is a right endpoint and the ($t+1$)-th node is $\clan{+}$, then there is a pipe going through the upper edge of this added box $A$.  Prolong this pipe as a vertical pipe $\BPD{\I}$ in the added box $A$. 
    $$
    \clan{2\dots.+\dots\dots}
    \xrightarrow{\ s_t*-\ }
    \clan{3\dots+.\dots\dots}\qquad 
    \text{e.g. }
    \BPD{\M{\dots}\M{\dots}\M{\dots}\\
    \M{\dots}\F\X\\\M{\dots}\J\M{A}}\longrightarrow \BPD{\M{\dots}\M{\dots}\M{\dots}\\
    \M{\dots}\F\X\\\M{\dots}\J\I}.
    $$

\item If the $t$-th node of $\gamma$ is $\clan{-}$ and the ($t+1$)-th node is a left endpoint, 
    then there is a pipe going through the left edge of this added box $A$. Prolong this pipe as a horizontal pipe $\BPD{\H}$ in the added box $A$. 
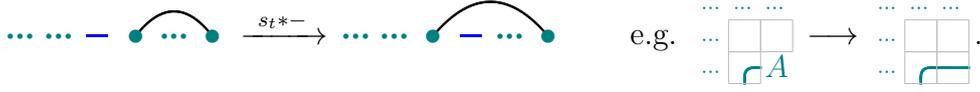
    $$\clan{\dots\dots-2\dots.}
    \xrightarrow{\ s_t*-\ }
    \clan{\dots\dots3-\dots.}\qquad 
    \text{e.g. }
    \BPD{\M{\dots}\M{\dots}\M{\dots}\\
    \M{\dots}\O\O\\\M{\dots}\F\M{A}}\longrightarrow \BPD{\M{\dots}\M{\dots}\M{\dots}\\
    \M{\dots}\O\O\\\M{\dots}\F\H}.
    $$
\item 
    If the $t$-th node of $\gamma$ is a right endpoint and the ($t+1$)-th node is a left endpoint, then there exist two pipes, going through the upper and left edge of the box $A$, respectively. 
    Prolong these two pipes to form a cross $\BPD{\X}$ in the added box $A$.
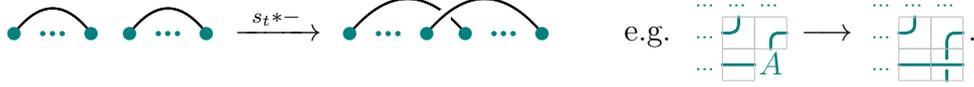
    $$\clan{2\dots.2\dots.}
    \xrightarrow{\ s_t*-\ }
    \clan{3\dots3.\dots.}
\qquad \text{e.g. }
    \BPD{\M{\dots}\M{\dots}\M{\dots}\\
    \M{\dots}\J\F\\\M{\dots}\H\M{A}}\longrightarrow \BPD{\M{\dots}\M{\dots}\M{\dots}\\
    \M{\dots}\J\F\\\M{\dots}\H\X}.
    $$
\end{itemize}

Note that this construction gives no additional empty tiles, hence $\s_{\gamma}(y) = \s_{s_t*\gamma}(y)$.
Since $v_{s_t*\gamma}=v_\gamma s_t<v_\gamma$, we have $s_t*\gamma>\gamma$ and by Proposition \ref{prop:localization} (4) we have $\left.\Upsilon_{\gamma}(x;y)\right|_{v_{\gamma}s_{t}}=0$.
According to Proposition \ref{lem:indforloc}, we have  
\begin{equation*}
   \left.\Upsilon_{\gamma}(x;y)\right|_{v_{\gamma}}=(y_{v_{\gamma}(t)}-y_{v_{\gamma}(t+1)})\left.\Upsilon_{s_t*\gamma}(x;y)\right |_{v_{\gamma}s_t}.
\end{equation*}
By the induction hypothesis,  we derive that
\begin{align*}
     \Upsilon_{\gamma}(x;y)|_{v_{\gamma}}&= (y_{v_{\gamma}(t)}-y_{v_{\gamma}(t+1)})     \Upsilon_{s_t\ast\gamma}(x;y)|_{v_{\gamma}s_t}\\
    &=(y_{v_{\gamma}(t)}-y_{v_{\gamma}(t+1)})\s_{s_t*\gamma}(y)\cdot \prod_{(i,j)\in \overline{\lambda(s_t*\gamma)}}(y_{n-j+1}-y_i)\\
    & =\s_{\gamma}(y)\cdot \prod_{(i,j)\in \overline{\lambda(\gamma)}}(y_{n-j+1}-y_i).
    \qedhere
\end{align*}
\end{proof}

\subsection{Proof of Theorem \ref{th:main}}

According to \cite[Proposition 2.6]{WY}, for any ($p,q$)-clan $\gamma$, the polynomial $\Upsilon_{\gamma}(x;y)$ can be expressed as a linear combination of double Schubert polynomials $\mathfrak{S}_w(x;y)$ with $w \in S_n$:
\begin{equation}\label{eq:expand2schubert}
    \Upsilon_{\gamma}(x;y)=\sum_{w\in S_n}c_{\gamma,w}(y)\S_{w}(x;y).
\end{equation}

\begin{proof}[Proof of Theorem 
\ref{th:main}]
Applying divided difference operator $\partial_u$ on both sides of equation \eqref{eq:expand2schubert}, we obtain
\begin{align*}
    \partial_{u}\Upsilon_{\gamma}(x;y)=\sum_{w\in S_n}c_{\gamma,w}(y)\partial_{u}\S_{w}(x;y). 
\end{align*}
Now taking localization at the identity permutation on both sides, combining the fact that $\left.\partial_{u}\S_{w}(X;Y)\right|_{\id}=1$ if $u=w$ and 0 otherwise, we obtain
\begin{align}\label{xishuc}
c_{\gamma,w}(y) = 
\begin{cases}
\Upsilon_{w*\gamma}(x;y)|_{\id},&    \text{if} \ 
\ell(w*\gamma)-\ell(\gamma)=\ell(w),\\
0,& \text{otherwise}.
\end{cases}
\end{align}
Now Theorem \ref{th:main} follows immediately from \eqref{xishuc}, Proposition \ref{prop:local at id} and Proposition \ref{prop:localization} (4). 
\end{proof}

\subsection{Proof of Theorem \ref{thm:mainLRcoeff}}

Recall that $X_w$ is the closure of the $B^-$-orbit of $V_\bullet^w=\dot{w}B/B$. 
Let $X^w$ be the closure of the $B$-orbit of $V_\bullet^w=\dot{w}B/B$. 
\begin{proof}[Proof of Theorem \ref{thm:mainLRcoeff}]
Note that we have 
$\int_{Fl_n} [X^u]_T\smile [X_v]_T=\delta_{uv}$. 
Taking Poincar\'e pairing with $[X^{w_0w}]_T$ on both sides of 
\[[X_{u_\gamma}]_T\smile[X_{v}]_T=\sum_{w\in S_n}
d_{u_\gamma,v}^{w}(y)[X_{w}]_T,\]
we obtain 
$$d_{u_\gamma,v}^{w_0w}(y) = 
\int_{Fl_n} [X_{u_\gamma}]_T\smile [X_v]_T\smile [X^{w_0w}]_T. 
$$

On the other hand, by \cite{Wyser}, $Y_{\gamma}$ is a Richardson variety if $\gamma$ is a non-crossing clan, namely, $[Y_\gamma]_T=[X_{v_\gamma}]_T\smile[X^{w_0u_\gamma}]_T$. Expand $[Y_\gamma]_T$ into Schubert classes we have
$$[X_{v_\gamma}]_T\smile[X^{w_0u_\gamma}]_T = \sum_{w\in S_n}c_{\gamma,w}(y)[X_w]_T,$$
and 
\[c_{\gamma,w_0v}(y)=\int_{Fl_n}
[X_{v_\gamma}]_T\smile[X^{w_0u_\gamma}]_T
\smile [X^{w_0v}]_T.\]
Now reverse the order of $y$ and note that $w_0Bw_0=B^-$, we obtain 
\begin{align*}
    c_{\gamma,w_0v}(\cev{y})&=\int_{Fl_n}
[w_0X_{v_\gamma}]_T\smile[w_0X^{w_0u_\gamma}]_T\smile [w_0X^{w_0v}]_T
\\
&=\int_{Fl_n}
[X^{w_0v_\gamma}]_T\smile[X_{u_\gamma}]_T\smile [X_{v}]_T
=d_{u_\gamma,v}^{w_0v_\gamma}(y).
\qedhere 
\end{align*}
\end{proof}

\section{Smoothness via localization}\label{sec:smoothness}

It is well known that localization can be used to characterize smoothness. For our purpose, we need the following result, which 
appears in the proof of  \cite[Corollary 19]{Brion2}. We include a proof here for completeness.

\begin{Lemma}\label{lem:smoothcri}
Let $Y \subseteq G/B$ be a closed $T$-equivariant subvariety. Then $Y$ is smooth at $u \in W$ if and only if  $[Y]_T|_{u} \in H_{T}^{*}(\mathrm{pt})$ satisfies
\begin{equation}\label{eq:smoothcri}
    [Y]_T|_{u} = \prod_{(a,b)} (y_{u(a)}-y_{u(b)}),
\end{equation}
where the product goes through all $T$-stable curves not in $Y$ from $u$ with weight $y_{u(a)}-y_{u(b)}$. 
\end{Lemma}

\begin{proof}
By  \cite[Theorem 17]{Brion2}, $Y$ is smooth if and only if  the equivariant multiplicity $e_u(Y)$ of $Y$ at $u$ is equal to
\begin{equation}\label{eq:smooth}
e_u(Y)=\prod_{T'}e_u(Y^{T'})=\prod_{\alpha}\frac{1}{u\alpha},
\end{equation}
where $\alpha$ are positive roots such that the $T$-stable curves connecting $u$ and $us_\alpha$ is in $Y$. 
On the other hand, since the weight of the tangent space at $u$ is $\prod_{\alpha>0}u\alpha$, we have 
\begin{align}\label{mlutiequiv}
e_u(Y)=\frac{[Y]_T|_u}{\prod_{\alpha>0}u\alpha}.
\end{align}
Combining \eqref{mlutiequiv} and \eqref{eq:smooth}, we are lead to  \eqref{eq:smoothcri}. 
\end{proof}

To consider the smoothness of the orbit closure $Y_{\gamma}$ of any clan $\gamma$ at $v_{\gamma}$, by \eqref{eq:smoothcri},  we need to discuss the $T$-stable curves starting from $v_\gamma$ not in $Y_{\gamma}$.
Recall that $T$-stable curves on $Fl_n$ are in bijection with cover relations in the Bruhat order, see \cite{Carrell94}. 
Precisely, for $u,w\in S_n$ with $w=ut_{ab}$ for some $a<b$, we can construct a curve 
\begin{equation}\label{eq:curve}
\eta_{u,w}:\mathbb{C}\longrightarrow Fl_n,\qquad 
z\longmapsto A(z)B/B,
\end{equation}
where the $i$-th column of $A(z)$ is $e_{u(i)}+z\delta_{ia}e_{u(b)}$. 
The morphism $\eta_{u,w}$ can be extended to $\mathbb{P}^1$ and $\eta_{u,w}(0)= V_\bullet^u=\dot{u}B/B$ and $c(\infty)=V_\bullet^{w}=\dot{w}B/B$. 
We will refer the curve $\eta_{u,w}$ to \emph{the $T$-stable curve connecting $u$ and $w$}, or 
\emph{the $T$-stable curve starting from $u$ with weight $y_{u(a)}-y_{u(b)}$}. 

\begin{Prop}\label{prop:minimalorbitcontcurves}
Assume $u<w\in S_n$ with $w=ut_{ab}$ for $a<b$ (in particular, we have $u(a)<u(b)$). 
Let $Y_\gamma$ be the minimal $K$-orbit closure contains the $T$-stable curve $\eta_{u,w}$. 
Then  $\gamma$ can be constructed as follows. Let $v$ be the unique $p$-inverse Grassmannian permutation in the right coset $(S_p\times S_q)u$.
\begin{enumerate}[(\romannumeral1)]
\item \label{it:case1} If $u(b)\leq p$ or $p<u(a)$, then $\gamma$ is the unique matchless clan with $v_\gamma = v$, i.e., 
\begin{equation*}
\text{the $i$-th node of $\gamma$}
=\begin{cases}
\clan{+}, & \text{if} \ u(i)\leq p,\\
\clan{-}, & \text{if} \ u(i)> p.\\
\end{cases}
\end{equation*}

\item If $u(a)\leq p<u(b)$, then $\gamma$ is the unique clan with exactly one $(a,b)$-matching and $v_\gamma=v$.
 
\end{enumerate}

\end{Prop}
\begin{proof}
Note that we can write the matrix $A(z)$ in \eqref{eq:curve} as 
$A(z)=B(z)\dot{u}$, where 
\[B(z)=\mathbf{1}_n+zE_{u(b),u(a)},\]
and $E_{u(b),u(a)}$ is the elementary matrix with $1$ in the $(u(b),u(a))$ position and 0 elsewhere.

In Case (i),  we have $B(z)\in GL_p\times GL_q$.
When $z\neq \infty$, $A(z)B/B$ is contained in the orbit
$K\dot{u}B/B$.
Since $\gamma$ is matchless, we have $\dot{v}_{\gamma}=\dot{v}=
\dot{\gamma}$. Thus the  orbit   $Y_\gamma^\circ$ contains $\dot{v}B/B$, or equivalently, $\dot{u}B/B$. Therefore, the orbit closure $Y_\gamma$ contains $A(z)$, i.e., the  $T$-stable curve connecting $u$ and $w$.

Let us consider Case (ii). 
Let $y\in S_p\times S_q$ be such that $yu=v$. 
Since $\dot{y}\in GL_p\times GL_q$,  the curve $\eta_{u,w}$ is contained in $Y_\gamma^\circ$ if and only if the curve $\eta_{yu,yw}$ is contained in $Y_\gamma^\circ$. Therefore, we can assume that $u=v$.  Since $\gamma$ has only one $(a,b)$-matching with $v_\gamma=v$, we have $\dot{\gamma}=\dot{v}_\gamma+E_{u(b),a}=\dot{v}+E_{u(b),a}$.
When $z\neq 0$, taking the $\{u(a),u(b)\}\times\{a,b\}$ minor, by the identity
$$\left[\begin{matrix}
1&0\\z&1
\end{matrix}\right]
= 
\left[\begin{matrix}
1&0\\0&z
\end{matrix}\right]
\left[\begin{matrix}
1&0\\1&1
\end{matrix}\right]
\left[\begin{matrix}
1&0\\0&z^{-1}
\end{matrix}\right], 
$$
we find  
$$A(z)B/B \in K\dot{\gamma}B/B=Y_\gamma^\circ.$$
This proves Case (ii).

Since $A(z)B/B\in K\dot{\gamma}B/B=Y_\gamma^\circ$ for any $z\in \mathbb{C}^\times$ in both cases, we find that any $Y_{\tau}$ containing the curve $\eta_{u,w}$ satisfies $Y_{\gamma}^{\circ}\subseteq Y_{\tau}$, so $Y_{\gamma}\subseteq Y_{\tau}$. That is, $Y_{\gamma}$ is the minimal.
\end{proof}

\begin{Th}\label{thm:descofcurveon}
Let $\gamma$ be a ($p,q$)-clan. 
Then the $T$-stable curve starting from $v=v_\gamma$ with weight $y_{v(a)}-y_{v(b)}$ ($a<b$) is contained in $Y_\gamma$ if and only if either one of the following holds:
\begin{itemize}
    \item[(1)] $\max\{v(a),v(b)\}\leq p$ or $\min\{v(a),v(b)\}>p$. 
    \item[(2)] $v(a)\leq p<v(b)$ with $i\leq a<b\leq j$ for some $(i,j)$-matching in $\gamma$. 
\end{itemize}
\end{Th}

To prove Theorem \ref{thm:descofcurveon}, we need to recall the partial order on $K$-orbit closures under containment which can be characterized by the Bruhat order on clans.

\begin{Prop}[{\cite[Theorem 1.2]{Wyser16}}]
\label{prop:wyser}
    Let $\gamma,\tau$ be two $(p,q)$-clans with $p+q=n$. 
    Then $Y_{\gamma}\subseteq Y_{\tau}$ if and only if the following hold for all $i<j$:
    $$\gamma(i;\clan{+})\ge\tau(i;\clan{+}),\qquad 
    \gamma(i;\clan{-})\ge\tau(i;\clan{-}),\qquad 
    \gamma(i;j)\le\tau(i;j),$$
        where 
\begin{align*}
    \gamma(i;\clan{+})
        & = \texttt{\#}\{\clan{+}\text{ and matchings in the first $i$ nodes of }\gamma\},\\
    \gamma(i;\clan{-})
        &=\texttt{\#}\{\clan{-}\text{ and matchings in the first $i$ nodes of }\gamma\},\\
    \gamma(i;j)&=\texttt{\#}\{(s,t)\text{-matching of }\gamma\text{ with }s\le i<j\leq t \}.
    \end{align*}
\end{Prop}

The partial order under containment on orbit closures defines a partial order, called the strong Bruhat order, on clans. That is, we can define $\gamma\preceq\tau$ in the strong Bruhat order if and only if $Y_\gamma\subseteq Y_{\tau}$.

\begin{Lemma}\label{lem:wyser1}
 In Proposition \ref{prop:wyser}, the condition
$\gamma(i;\clan{+})\geq \tau(i;\clan{+})$ for all $i\in [n]:=\{1,2,\ldots,n\}$ is equivalent to $u_\gamma\ge u_\tau$ under Bruhat order. 
Similarly, 
$\gamma(i;\clan{-})\geq \tau(i;\clan{-})$ for all $i\in [n]$ 
is equivalent to $v_\gamma\ge v_\tau$. 
\end{Lemma}
\begin{proof}
Notice that $\gamma(i;\clan{+})$ is also equal to the number of $\clan{+}$ and right endpoints in the first $i$ nodes of $\gamma$.
Let $\gamma'$ (resp., $\tau'$) be the matchless clan obtained from $\gamma$ (resp., $\tau$) by replacing each left endpoint  with $\clan{-}$ and right endpoint with $\clan{+}$. Then we have $\gamma'(i;\clan{+})=\gamma(i;\clan{+})$ and $u_{\gamma}=u_{\gamma'}$. Therefore, 
\[\gamma(i;\clan{+})\ge \tau(i;\clan{+})\Leftrightarrow\gamma'(i;\clan{+})\ge \tau'(i;\clan{+})\Leftrightarrow \gamma'(i;\clan{-})\le \tau'(i;\clan{-}), \text{for\ } i\in [n].\]
 For any two $q$-inverse Grassmannian permutations $w_1,w_2$, we have $w_1\leq w_2$ under Bruhat order $\Leftrightarrow$
$w_{1}^{-1}(i)\le w_{2}^{-1}(i)$ for  $i\in[q]$ $\Leftrightarrow$ $\texttt{\#}(w_{1}[i]\cap [q])\ge \texttt{\#}(w_{2}[i]\cap [q])$ for $i\in [n]$. Thus 
\[\gamma'(i;\clan{-})\le \tau'(i;\clan{-})\Leftrightarrow  \texttt{\#}(u_{\gamma'}[i]\cap [q])\le \texttt{\#}(u_{\tau'}[i]\cap [q])\Leftrightarrow u_{\gamma'}\ge u_{\tau'}.\]
Similarly, $\gamma'(i;\clan{-})\ge \tau'(i;\clan{-})$ is equivalent to $\tau'$ has as least as many $\clan{+}$'s than $\gamma'$ among the first $i$ positions, that is, $v_{\gamma'}\ge v_{\tau'}$.
\end{proof}

\begin{Coro}\label{cor:contain3}
Let $\gamma$ be a ($p,q$)-clan. 
If the torus fixed point $V^w_\bullet=\dot{w}B/B\in Y_\gamma$, then 
$v_\gamma\leq w$. 
\end{Coro}
\begin{proof}
Let $\tau$ be the matchless clan such that $v_{\tau}$ is the minimal element in $(S_p\times S_q)w$. 
Then $v_{\tau}\le w$ and  $\dot{v}_{\tau}B/B\in Y_{\tau}^\circ$. 
Since $w=yv_\tau$ for some $y\in S_p\times S_q$, we have  $\dot{w}B/B\in Y^{\circ}_{\tau}$. 
Since $\dot{w}B/B\in Y_{\gamma}$, we have $Y_{\tau}\subseteq Y_{\gamma}$.
By Proposition \ref{prop:wyser}, we have $\tau(i;\clan{-})\ge \gamma(i;\clan{-})$ for $i\in[n]$ and by Lemma \ref{lem:wyser1}, we have $w\geq v_\tau\geq v_\gamma$. 
\end{proof}

We need two more corollaries which are immediate from \cite[Theorem 2.8]{Wyser16}. 
\begin{Coro}\label{cor:contain1}
Let $\gamma$ be a $(p,q)$-clan with a $\clan{+}$ at the $i$-th node and a $\clan{-}$ at the $j$-th node, $i<j$ . Let $\gamma'$ be the clan obtained from $\gamma$ by matching the $i$-th and $j$-th node. Then $Y_{\gamma}\subseteq Y_{\gamma'}$. 
    $$\gamma = \clan{\dots+\dots\dots-\dots}\qquad 
    \gamma' = \clan{\dots3\dots\dots.\dots}$$
\end{Coro}

\begin{Coro}\label{cor:contain2}
    Let $\gamma$ and $\tau$ be two clans with 
    $v_\gamma=v_\tau$. 
    If $\gamma$ has a unique $(i_1,j_1)$-matching, $\tau$ has a unique $(i_2,j_2)$-matching, and $i_1\le i_2\le j_2\le j_1$, then $Y_{\tau}\subseteq Y_{\gamma}$. 
$$\begin{matrix}
\tau = \clan{\dots+\dots2\dots.\dots-\dots}\qquad 
\gamma = \clan{\dots6\dots+\dots-\dots.\dots}\\
\end{matrix}$$
\end{Coro}

Now we are ready to prove Theorem \ref{thm:descofcurveon}.

\begin{proof}[Proof of Theorem \ref{thm:descofcurveon}]
We first show that the curves satisfying  conditions in  Case (1) or Case (2) must lie in $Y_{\gamma}$.
By Proposition \ref{prop:minimalorbitcontcurves} (i), the curves in Case (1) are contained in $Y_\tau$, where $\tau$ is the matchless clan with $v_\tau = v_\gamma$.
By Corollary \ref{cor:contain1}, the closed orbit $Y_{\tau}$ is contained in $ Y_{\gamma}$. 
So the curves in Case (1) are contained in $Y_{\gamma}$.

By Proposition \ref{prop:minimalorbitcontcurves} (ii), the curves in Case (2) are contained in some $Y_\tau$, such that $\tau$ has only one $(a,b)$-matching with $v_\tau = v_\gamma$.
By Corollary \ref{cor:contain2}, if $\gamma'$  has only one  $(i,j)$-matching and $v_{\gamma'}=v_{\tau}$, then 
$Y_{\tau}\subseteq Y_{\gamma'}$.  
By Corollary \ref{cor:contain1}, if $\gamma$ has an $(i,j)$-matching, then $Y_{\gamma'}\subseteq Y_{\gamma}$. 
Thus the curves in Case (2) are contained in $Y_\gamma$. 

Conversely, assume the $T$-stable curve $\eta_{v,w}$ from $v=v_\gamma$ to $w=vt_{ab}$ is contained in $Y_\gamma$. 
If $v(b)\leq p<v(a)$, then $w=vt_{ab}<v$. 
By Corollary \ref{cor:contain3}, $\dot{w}B/B\notin Y_\gamma$, so the curve $\eta_{v,w}$  cannot be contained in $Y_\gamma$.
Thus we can assume $v(a)\leq p<v(b)$. By Proposition~\ref{prop:minimalorbitcontcurves} (ii), there is a clan $\tau$ with only one $(a,b)$-matching such that $Y_\tau$ is the minimal orbit closure which contains the curve $\eta_{v,w}$. Thus $Y_\tau\subseteq Y_\gamma$. 
By Proposition \ref{prop:wyser}, 
$1=\tau(a,b)\leq \gamma(a,b)$. 
In particular, there must be some $(i,j)$-matching in $\gamma$ with $i\leq a<b\leq j$. 
\end{proof}

 Finally, we are in a position to prove Theorem \ref{th:smoothness}.

\begin{proof}[Proof of Theorem \ref{th:smoothness}]
(1) $\Leftrightarrow$ (2). 
By Theorem \ref{thm:descofcurveon}, 
the $T$-stable curve starting from $v=v_\gamma$ with weight $y_{v(a)}-y_{v(b)}$ ($a<b$) is NOT contained in $Y_\gamma$ if and only if either one of the following  two cases holds:
\begin{enumerate}
    \item[(I)] $v(b)\leq p<v(a)$ or 
    \item[(II)] $v(a)\leq p<v(b)$ but no $(i,j)$-matching in $\gamma$ such that
    $i\leq a<b\leq j$. 
\end{enumerate}
Recall that $v_\gamma$ is obtained by assigning each $\clan{+}$ or left endpoint  with $1,2,\ldots,p$ from left to right, and assigning $\clan{-}$ or right endpoint  with $p+1,p+2,\ldots,n$ from left to right. 
Note that if the $i$-th node of $\gamma$ is $\clan{+}$ or left endpoint (resp., $\clan{-}$ or right endpoint), then the $i$-th step on the southeast boundary of $\lambda(\gamma)$ lies in the $v_\gamma(i)$-th row (resp., $(n+1-v_{\gamma}(i))$-th column).

Case (I) implies that the $a$-th node of $\gamma$ is a $\clan{-}$ or a right endpoint and the  $b$-th node is a $\clan{+}$ or left endpoint. Since $a<b$, the tile at $(v(b),n+1-v(a))$ is in $\overline{\lambda(\gamma)}$. See the $*$ tiles in \eqref{eq:curves not in} as an example. Therefore, we have
$$\prod_{(a,b)\text{ in Case (I)}}(y_{v(a)}-y_{v(b)})
=\prod_{(i,j)\in \overline{\lambda(\gamma)}}
(-y_i+y_{n-j+1}).$$

For Case (II), the $a$-th node of $\gamma$ is a $\clan{+}$ or a left endpoint and the  $b$-th node is a $\clan{-}$ or right endpoint. Since $a<b$, the tile at $(v(b),n+1-v(a))$ belongs to $\lambda(\gamma)$. 
Because no $(i,j)$-matching exists for $i \le a < b \le j$, there are no pipes northwest to the box at $(v(a),n+1-v(b))\in \lambda(\gamma)$. Thus the product of weights in Case (II) corresponds to $$\prod_{(a,b)\text{ in Case (II)}}(y_{v(a)}-y_{v(b)})
=\prod_{(i,j)}(y_i-y_{n+1-j}),$$
where the product runs over empty tiles $(i,j)\in \lambda(\gamma)$  such that there are no pipes northwest of $(i,j)$ in the Rothe BPD $D(\gamma)$. 
See the shadowed tiles in \eqref{eq:curves not in} for an example. 

By Lemma \ref{lem:smoothcri}, $Y_\gamma$ is smooth at $v_\gamma$ if and only if 
$$[Y_\gamma]_T|_{v_\gamma} = \prod_{(a,b)\  \text{in Case (I) or (II)}}(y_{v(a)}-y_{v(b)}). $$
The equivalence of (1) and (2) follows from 
Theorem \ref{th:localization} 
and the definition of $\mathfrak{s}_\gamma(x)$. 
 
(2) $\Leftrightarrow$ (3).
Note that if $\tau$ appears as a pattern of $\gamma$, then  $D(\tau)$ can be obtained from  $D(\gamma)$ by deleting some rows and columns. 
So if $D(\gamma)$ contains any of the 5 patterns in Theorem \ref{th:smoothness}, then there exists an empty tile southeast to a pipe, thus we can obtain another BPD fragment by applying a droop operation. 
This implies (2) $\Rightarrow$ (3). 
$$\begin{array}{ccccc}
\BPD{\F\H\\\I\o}& 
\BPD{\F\H\\\I\o\\\I\F}& 
\BPD{\F\H\H\\\I\o\F}& 
\BPD{\F\H\H\\\I\o\F\\\I\F}&
\BPD{\F\H\H\\\I\o\F\\\I\F\X} 
\\[-0.5pc]\\
\clan{3+-.}& \clan{4+1..}& \clan{41.-.}& \clan{51.1..}& \clan{522...}.
\end{array}$$
 
Conversely, we show that if $\gamma$ has more than one BPD fragment, then $\gamma$ must contains one of the 5 patterns. By Proposition \ref{prop:Anna BPD can be obtained by droops}, any BPD fragment can be obtained from $D(\gamma)$  by applying droop operations. Thus there exists some pipe  with an empty tile $\BPD{\O}$ southeast to this pipe in $D(\gamma)$. Suppose that such a $\BPD{\O}$ is at position
$(v_\gamma(a),n+1-v_\gamma(b))$. Then the $a$-th node of $\gamma$ must be a $\clan{+}$ or the left endpoint of some $(a,j)$-matching with $j<b$
(since if $j\geq b$, then the box at $(v_\gamma(a),n+1-v_\gamma(b))$ will be $\BPD{\H}$ or $\BPD{\F}$). Similarly, the $b$-th node of $\gamma$ must be a $\clan{-}$ or the right endpoint of some $(i,b)$-matching with $i>a$. There are five possibilities illustrated below: 
$$\begin{array}{ccccc}
\BPD{
\F\H\X\H\X\\
\I\o\I\O\I\\
\I\O\I\\
\X\H\X\\
\I}& 
\BPD{
\F\H\H\X\H\\
\I\o\O\I\O\\
\I\F\H\\
\X\X\H\\
\I}& 
\BPD{
\F\H\X\H\X\\
\I\o\I\F\X\\
\X\H\X\\
\I\O\I\\
\I}& 
\BPD{
\F\H\H\H\X\\
\I\o\O\F\X\\
\I\F\H\\
\X\X\H\\
\I}& 
\BPD{
\F\H\H\X\H\\
\I\o\F\X\H\\
\X\H\X\\
\I\F\X\\
\I}
\\[-0.5pc]\\
\clan{3+-.}& \clan{4+1..}& \clan{41.-.}& \clan{51.1..}& \clan{522...}.
\end{array}$$
This completes the proof of (3) $\Rightarrow$ (2). 
\end{proof}

\begin{Eg}
Let us consider the running example \eqref{eq:egofugammavgamma}. 
The weights of curves NOT contained in $Y_\gamma$ are 
$$\begin{array}{rl}
\text{\rm(I)}& y_{6}-y_{2},\,y_6-y_3,\,y_6-y_4,\,y_6-y_5,\,y_7-y_5,\,y_8-y_5,\,y_9-y_5,\,y_{10}-y_5;\\
\text{\rm(II)}& y_1-y_{11},\, y_1-y_{10},\, y_1-y_9.
\end{array}$$
The Rothe BPD fragment is 
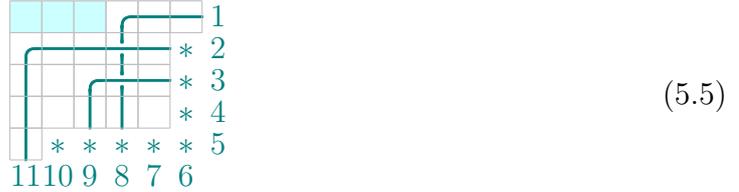
\begin{equation}\label{eq:curves not in}
\BPD{
\o\o\o\F\H\H\M{1}\\
\F\H\H\X\H\M{*}\M{2}\\
\I\O\F\X\H\M{*}\M{3}\\
\I\O\I\I\O\M{*}\M{4}\\
\I\M{*}\M{*}\M{*}\M{*}\M{*}\M{5}\\
\M{11}\M{10}\M{9}\M{8}\M{7}\M{6}}
\end{equation}
Here the shadowed tiles correspond to Case (II), and the $*$ marks the Case (I).
\end{Eg}

Finally, we remark that, unlike the case of Richardson varieties, the fact that the  $K$-orbit closure $Y_\gamma$ is smooth at $v_\gamma$ and $w_0u_\gamma$ does not guarantee that $Y_\gamma$ is smooth. 
 
\begin{Eg}
Let $\gamma=\clan{22..}$. 
The orbit closure admits the following explicit description
$$Y_\gamma = \{V_\bullet\in Fl_4
\mid XV_1\subset V_3\}$$
where $X=\operatorname{diag}(1,1,0,0)$; see \cite[Corollary 3.7]{CPSU}. 
Note that $Y_\gamma $ is a $\mathbb{P}^1$-bundle over 
$Z = \{V_1\subset V_3\mid XV_1\subset V_3\}$. 
By Pl\"ucker embedding, $Z$ can be identified 
as the subvariety of 
$\mathbb{P}^3\times \mathbb{P}^3$
cut out by the equations 
$$\begin{cases}
x_1y_1+x_2y_2=0,\\ 
x_3y_3+x_4y_4=0,
\end{cases}\qquad 
\text{where }
\begin{array}{r@{\,}l}
{}[x_1:x_2:x_3:x_4]&\in \mathbb{P}^3,\\
{}[y_1:y_2:y_3:y_4]&\in \mathbb{P}^3.
\end{array}$$
The singular loci can be described as 
$$\{x_1=x_2=y_1=y_2=0\}\cup \{x_3=x_4=y_3=y_4=0\}.$$
Using this, it is easy to check $Y_\gamma$ is nonsingular at $\dot{w}B/B$ for $w\in S_4$ with
$$
w(1)\notin \{1,2\}\ni w(4)\quad \text{or}\quad 
w(4)\notin \{1,2\}\ni w(1). $$
In particular, $Y_\gamma$ is smooth at $v_\gamma =\id$. 
Similarly, $Y_\gamma$ is smooth at $w_0u_\gamma$.
However, $Y_\gamma$ itself is not smooth. 
\end{Eg}

\appendix
\section{Examples of ($2,2$)-clans}\label{appendix}

The following diagram displays all ($2,2$)-clans. 
$$\def\myclan#1{\begin{matrix}
\clan{#1}
\end{matrix}}
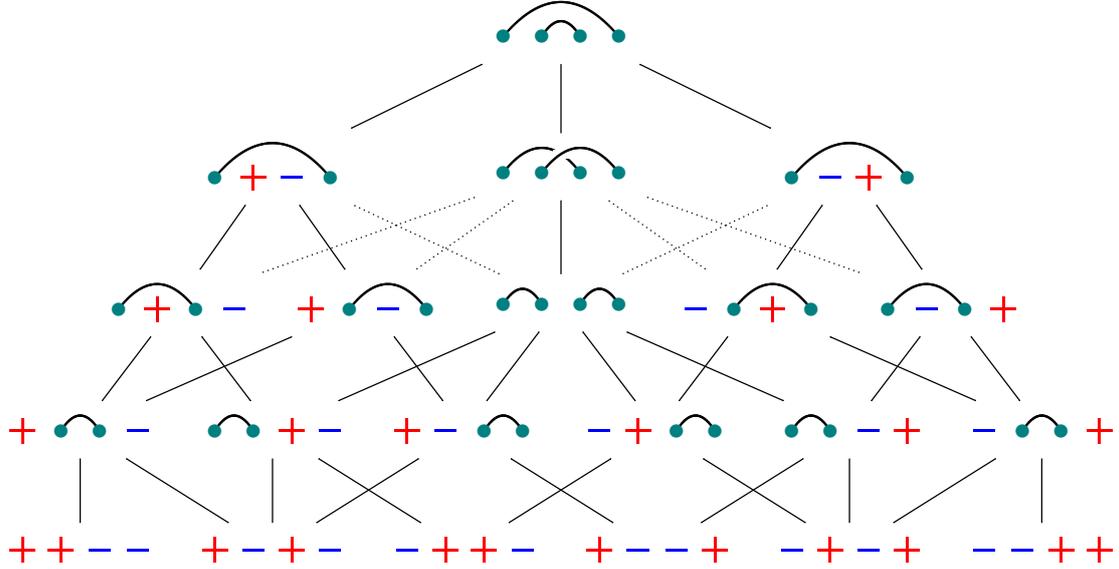
\xymatrix@!C=1pc{
&&&&&\myclan{31..}\\
&&
{\myclan{3+-.}}
    \ar@{-}[urrr]&&&
{\myclan{22..}}
    \ar@{-}[u]&&&
{\myclan{3-+.}}
    \ar@{-}[ulll]
\\&
{\myclan{2+.-}}
    \ar@{-}[ur]\ar@{..}[urrrr]&&
{\myclan{+2-.}}
    \ar@{-}[ul]\ar@{..}[urr]&&
{\myclan{1.1.}}
    \ar@{-}[u]\ar@{..}[ulll]\ar@{..}[urrr]&&
{\myclan{-2+.}}
    \ar@{-}[ur]\ar@{..}[ull]&&
{\myclan{2-.+}}
    \ar@{-}[ul]\ar@{..}[ullll]
\\%
{\myclan{+1.-}}
    \ar@{-}[ur]\ar@{-}[urrr]&&
{\myclan{1.+-}}
    \ar@{-}[ul]\ar@{-}[urrr]&&
{\myclan{+-1.}}
    \ar@{-}[ul]\ar@{-}[ur]&&
{\myclan{-+1.}}
    \ar@{-}[ul]\ar@{-}[ur]&&
{\myclan{1.-+}}
    \ar@{-}[ur]\ar@{-}[ulll]&&
{\myclan{-1.+}}
    \ar@{-}[ul]\ar@{-}[ulll]
\\%
{\myclan{++--}}
    \ar@{-}[u]&&
{\myclan{+-+-}}
    \ar@{-}[u]\ar@{-}[ull]\ar@{-}[urr]&&
{\myclan{-++-}}
    \ar@{-}[ull]\ar@{-}[urr]&&
{\myclan{+--+}}
    \ar@{-}[ull]\ar@{-}[urr]&&
{\myclan{-+-+}}
    \ar@{-}[u]\ar@{-}[ull]\ar@{-}[urr]&&
{\myclan{--++}}
    \ar@{-}[u]
}$$
Here solid lines stand for the cover relations for the weak order. Solid and dotted lines stand for the cover relations for the strong order in Proposition \ref{prop:wyser}. 
\def\term#1#2#3#4{
{\makebox[0.15\linewidth]{\(#1\)}}
\makebox[0.1\linewidth]{\(\sf #2\)}\quad
\makebox[0.5\linewidth][l]{\(#3\)}
{\makebox[0.2\linewidth][l]{\(#4\)}}
}
\[\term{\text{clans $\gamma$}}{\mbox{\(v_\gamma\)}}
{[Y_\gamma]_T|_{v_\gamma}=\Upsilon(x;y)|_{v_\gamma}}{\text{BPD fragments}}\]
\[\term{\clan{++--}}{1234}
{(y_{1} - y_{3}) (y_{2} - y_{3}) (y_{2} - y_{4})  (y_{1} - y_{4})}
{\BPD{\O\O\\\O\O}}\]
\[\term{\clan{+1.-}}{1234}
{(y_{1} - y_{3}) (y_{2} - y_{4}) (y_{1} - y_{4})}{
\BPD{\O\O\\\O\F}
}\]
\[\term{\clan{2+.-}}{1234}
{(y_{2} - y_{4}) (y_{1} - y_{4})}{
\BPD{\O\F\\\O\I}
}\]
\[\term{\clan{+2-.}}{1234}
{(y_{1} - y_{4})  (y_{1} - y_{3})}{
\BPD{\O\O\\\F\H}
}\]
\[\term{\clan{22..}}{1234}
{(y_{1} - y_{4})}{
\BPD{\O\F\\\F\X}
}\]
\[\term{\clan{3+-.}}
{1234}{(y_{1} + y_{2} - y_{3} - y_{4})}{
\BPD{\F\H\\\I\O}
\BPD{\O\F\\\F\J}
}\]
\[\term{\clan{31..}}{1234}
{1}{
\BPD{\F\H\\\I\F}
}\]
\[\term{\clan{+-+-}}{1324}
{(-y_{2} + y_{3})\cdot (y_{1} - y_{3}) (y_{2} - y_{4})(y_{1} - y_{4})}{
\BPD{\O\O\\\O}
}\]
\[\term{\clan{1.+-}}{1324}
{(-y_{2} + y_{3}) \cdot (y_{2} - y_{4}) (y_{1} - y_{4})}
{\BPD{\O\F\\\O}}
\]
\[\term{\clan{+-1.}}{1324}
{(-y_{2} + y_{3}) \cdot (y_{1} - y_{4}) (y_{1} - y_{3})}{
\BPD{\O\O\\\F}
}\]
\[\term{\clan{1.1.}}{1324}
{(-y_{2} + y_{3}) \cdot (y_{1} - y_{4})}{
\BPD{\O\F\\\F}
}\]
\[\term{\clan{3-+.}}{1324}
{(-y_{2} + y_{3})}{
\BPD{\F\H\\\I}
}\]
\[\term{\clan{+--+}}{1342}
{(-y_{2} + y_{3})(-y_{2} + y_{4}) \cdot (y_{1} - y_{3})(y_{1} - y_{4})}{
\BPD{\O\O\\\|}
}\]
\[\term{\clan{1.-+}}{1342}
{(-y_{2} + y_{3})(-y_{2} + y_{4}) \cdot (y_{1} - y_{4})}{
\BPD{\O\F\\\|}
}\]
\[\term{\clan{2-.+}}{1342}
{(-y_{2} + y_{3})(-y_{2} + y_{4})}{
\BPD{\F\H\\\|}
}\]
\[\term{\clan{-++-}}{3124}
{(-y_{1} + y_{3}) (-y_{2} + y_{3})\cdot (y_{1} - y_{4})(y_{2} - y_{4})}{
\BPD{\O\^\\\O\M{}}
}\]
\[\term{\clan{-+1.}}{3124}
{(-y_{1} + y_{3})(-y_{2} + y_{3}) \cdot (y_{1} - y_{4})}{
\BPD{\O\^\\\F\M{}}
}\]
\[\term{\clan{-2+.}}{3124}
{(-y_{1} + y_{3})(-y_{2} + y_{3})}{
\BPD{\F\^\\\I\M{}}
}\]
\[\term{\clan{-+-+}}{3142}
{(-y_{1} + y_{3}) (-y_{2} + y_{3}) (-y_{2} + y_{4}) \cdot (y_{1} - y_{4})}{
\BPD{\O\^\\\|}
}\]
\[\term{\clan{-1.+}}{3142}
{(-y_{1} + y_{3}) (-y_{2} + y_{3}) (-y_{2} + y_{4})}{
\BPD{\F\^\\\|}
}\]
\[\term{\clan{--++}}{3412}
{(-y_{1} + y_{3}) (-y_{2} + y_{3})(-y_{1}+ y_{4})(-y_{2} + y_{4})}{\BPD{\|\^\M{}\^\\\|}
}\]

\end{document}